\documentclass[11pt,reqno]{amsart}
\usepackage[dvips]{graphicx,color}
\usepackage{amssymb,amscd,latexsym,epsfig,xypic}
\usepackage[mathscr]{euscript}

\DeclareFontFamily{OT1}{rsfs}{}
\DeclareFontShape{OT1}{rsfs}{n}{it}{<-> rsfs10}{}
\DeclareMathAlphabet{\curly}{OT1}{rsfs}{n}{it}

\newcommand\I{\curly I}

\renewcommand\O{\mathcal O}
\newcommand\PP{\mathbb P}

\newcommand\C{\mathbb C}
\newcommand\Q{\mathbb Q}

\newcommand\Z{\mathbb Z}
\newcommand\M{\mathcal M}
\newcommand\m{\mathfrak m}
\newcommand\T{\mathcal T}
\renewcommand\S{\mathcal S}

\newcommand{\Rt}[1]{\stackrel{#1\,}{\longrightarrow}}

\newcommand\into{\hookrightarrow}
\newcommand\Into{\ar@{^{ (}->}[r]}

\newcommand\udot{^{\bullet}}

\newcommand\take{\backslash}

\newfont{\bigtimesfont}{cmsy10 scaled \magstep5}
\newcommand{\bigtimes}{\mathop{\lower0.9ex\hbox{\bigtimesfont\symbol2}}}

\newcommand\im{\operatorname{im}}
\newcommand\id{\operatorname{id}}

\newcommand\Hom{\operatorname{Hom}}
\newcommand\Onto{\operatorname{Hom}^{\operatorname{onto}}}
\newcommand\Pure{\operatorname{Ext}^{\operatorname{pure}}}

\renewcommand\hom{\operatorname{hom}}
\newcommand\Ext{\operatorname{Ext}}
\newcommand\eext{\curly Ext}
\newcommand\ext{\operatorname{ext}}
\newcommand\Aut{\operatorname{Aut}}

\newcommand\Proj{\operatorname{Proj}\,}

\newcommand\Hilb{\operatorname{Hilb}}

\newcommand\Gr{\operatorname{Gr}}

\newcommand{\quot}{\operatorname{Quot}}
\newcommand{\bq}{\mathbf{Q}}
\newcommand{\bp}{\mathbf{P}}
\newcommand{\bb}{\mathbf{B}}

\newcommand{\Sl}{\operatorname{SL}}
\newcommand\beq[1]{\begin{equation}\label{#1}}
\newcommand\eeq{\end{equation}}
\newcommand\beqa{\begin{eqnarray*}}
\newcommand\eeqa{\end{eqnarray*}}

\makeatletter \@addtoreset{equation}{section} \makeatother

\newtheorem{defn}[equation]{Definition}
\newtheorem{thm}[equation]{Theorem}
\newtheorem{lem}[equation]{Lemma}

\newtheorem{prop}[equation]{Proposition}
\newtheorem{conj}[equation]{Conjecture}
\newenvironment{rmk}{\noindent\textbf{Remark}.}{\\}

\newenvironment{exa}{\noindent\textbf{Example}.}{\\}

\title[Wall crossing]{Hilbert schemes and stable pairs: GIT and derived category wall crossings}
\author{J. Stoppa and R. P. Thomas}

\begin{document}
\begin{abstract} \noindent 
We show that the Hilbert scheme of curves and Le Potier's moduli space of stable pairs with one dimensional support have a common GIT construction. The two spaces correspond to chambers on either side of a wall in the space of GIT linearisations.

We explain why this is not enough to prove the ``DT/PT wall crossing conjecture" relating the invariants derived from these moduli spaces when the underlying variety is a 3-fold. We then give a gentle introduction to a small part of Joyce's theory for such wall crossings, and use it to give a short proof of an identity relating the Euler characteristics of these moduli spaces. 

When the 3-fold is Calabi-Yau the identity is the Euler-characteristic analogue of the DT/PT wall crossing conjecture, but for general 3-folds it is something different, as we discuss.
\end{abstract}

%%%%%%%%%%%%%%%%%%%%%%%%%%%%%%%%%%%%%%%%%%%%%%%%%%%%%%%%%%%%%%%%%%%%%%%%%%%

\maketitle
\setcounter{tocdepth}{1}
\renewcommand\contentsname{\vspace{-1cm}}
\tableofcontents
\vspace{-1cm}

\section{Introduction}

This paper is motivated by the conjectural equivalence between two curve counting theories on smooth complex projective threefolds $X$: the one studied in \cite{MNOP} and the stable pairs of \cite{PT1}. 

MNOP and stable pairs invariants are sheaf-theoretic analogues of Gromov-Witten invariants, sometimes called DT and PT invariants respectively. The space of stable maps to $X$ is replaced by suitable moduli spaces of sheaves supported on curves in $X$. 

Fix $\beta\in H_2(X,\Z)$ and $n \in \mathbb{Z}$.
In MNOP theory we integrate suitable classes against the virtual fundamental class of the Hilbert scheme $I_n(X, \beta)$ of subschemes $Z$ of $X$ in the class $[Z]=\beta$ with holomorphic Euler characteristic $\chi(\O_Z)=n$. The virtual fundamental class comes from thinking of $I_n(X, \beta)$ as a moduli space of sheaves of trivial determinant -- namely the ideal sheaves $\I_Z$ with Chern character
\beq{cc}
\Big(1,0,-\beta,-n+\frac{\beta . c_1(X)}2\Big)\ \in\ H^0(X)\oplus H^2(X)\oplus H^4(X)\oplus H^6(X).
\eeq

For stable pair theory we work instead with \emph{stable pairs} $(F,s)$: $F$ is a \emph{pure} sheaf on $X$ with Chern character $(0,0,\beta,-n + \beta.c_1(X)/2)$, and $s\colon\O_X \to F$ is a section with
0-dimensional cokernel. A special case of the work of Le Potier \cite{lepot} constructs
the fine moduli space $P_n(X, \beta)$ as a projective scheme. The virtual fundamental class comes from thinking \cite{PT1} of $P_n(X, \beta)$ as a moduli space of objects of the derived
category of coherent sheaves on $X$ (with trivial determinant)  -- namely the complexes
$I\udot:=\{\O_X\to F\}$ with Chern character \eqref{cc}.

Roughly speaking, we think of $I_n(X,\beta)$ as parameterising pure curves plus points (free and embedded) on $X$. Any $Z\in I_n(X,\beta)$ contains a maximal
Cohen-Macaulay curve $C\subseteq Z$ (the pure curve: recall that Cohen-Macaulay means no embedded points) such that the kernel of $\O_Z\to\O_C$ is 0-dimensional (the points). Equally loosely we think of stable pairs as parameterising Cohen- Macaulay curves (the support of the sheaf $F$) and free points \emph{on the curve} (the cokernel of the section $s$).

Over the Zariski-open subset of Cohen-Macaulay curves $C$ with no free or embedded points, the moduli spaces $I_n(X,\beta)$ and $P_n(X,\beta)$ are isomorphic: the stable pair $\O_X\to\O_C$ determines and is determined by the kernel ideal sheaf $\I_C$. Indeed $I\udot$ is quasi-isomorphic
to $\I_C$.

When $\omega_X\cong\O_X$, i.e. $X$ is a Calabi-Yau threefold, MNOP and stable pair invariants take a particularly simple form. Then the virtual dimension is zero and we get invariants by taking the length of the 0-dimensional virtual cycle:
$$
I^{vir}_{m,_\beta}=\int_{[I_m(X,\beta)]^{vir}}1,
$$
and
$$ 
P^{vir}_{m,_\beta}=\int_{[P_m(X,\beta)]^{vir}}1.
$$

In this case the deformation-obstruction theories \cite{ThCasson, PT1, HT} used to define the virtual cycles are \emph{self dual} in the sense of \cite{BehrendDT}. This implies that $I^{vir}_{m,\beta},\,P^{vir}_{m,\beta}$ are in fact weighted Euler characteristics:
$$
I^{vir}_{m,\beta}=e(I_m(X,\beta),\chi^B), \qquad P^{vir}_{m,\beta}=e(P_m(X,\beta),\chi^B).
$$
Here the weighting function is Behrend's integer-valued constructible function $\chi^B$ \cite{BehrendDT}, which assigns to each point of the moduli space the multiplicity with which it contributes to the invariants. At smooth points of the moduli space, $\chi^B\equiv(-1)^{\dim}$.

We can also form their generating series
$$
Z^{I,vir}_\beta(X)(t):=\sum_{m\in\Z}I^{vir}_{m,\beta}t^m \qquad\text{and}\qquad
Z^{P,vir}_\beta(X)(t):=\sum_{m\in\Z}P^{vir}_{m,\beta}t^m.
$$
The conjectural equivalence between the MNOP and stable pair invariants in the Calabi-Yau case is then the following.

\begin{conj}\label{conj}\emph{\cite{PT1}} For $X$ a Calabi-Yau threefold,
$$
Z^{P,vir}_\beta(X)=\frac{Z^{I,vir}_\beta(X)}{Z^{I,vir}_0(X)}\,.
$$
Equivalently, for each $m\in\Z$ we have the following identity (where the right hand side is a finite sum):
\beq{virid}
I^{vir}_{m, \beta}\ =\ P^{vir}_{m, \beta}\ +\ I^{vir}_{1,0}\cdot P^{vir}_{m-1, \beta}\ +\ I^{vir}_{2,0}\cdot P^{vir}_{m-2,\beta}\ +\ \ldots\ .
\eeq
\end{conj}

Here $Z^{I,vir}_0(X)$ is the generating series of virtual counts of zero dimensional subschemes of $X$. By \cite{MNOP, BF, LP, Li} it is in fact
$$
Z^{I,vir}_0(X)(t)=M(-t)^{e(X)},
$$
where $M(t)$ is the MacMahon function, the generating function for 3-dimensional partitions.

Using Kontsevich-Soibelman's identities for $\chi^B$ \cite{KS}, now proved in some cases \cite{JS}, it should now be possible to extend what follows to the weighted Euler characteristics $I^{vir}_{m,\beta}$ and $P^{vir}_{m,\beta}$. But in this paper we content ourselves with working with the unweighted Euler characteristics
$$
I_{m,\beta}:=e(I_m(X,\beta)) \qquad\text{and}\qquad P_{m,\beta}:=e(P_m(X,\beta)),
$$
which are not deformation invariant. Form their generating series
$$
Z^I_\beta(X)(t):=\sum_{m\in\Z}I_{m,\beta}t^m \qquad\text{and}\qquad
Z^P_\beta(X)(t):=\sum_{m\in\Z}P_{m,\beta}t^m.
$$
In Sections \ref{firstProof} and \ref{secondProof} we will give two different proofs of the following topological analogue of Conjecture \ref{conj} (first proved
by Toda \cite{TodaPre} in the Calabi-Yau case, as discussed below).

\begin{thm}\label{main} Let $X$ be a smooth projective threefold. Then
$$
Z^P_\beta(X)=\frac{Z^I_\beta(X)}{Z^I_0(X)}\,.
$$
Equivalently for each $m\in\Z$ we have the following identity (where the right hand side is a finite sum):
\[I_{m,\beta}\ =\ P_{m, \beta}\ +\ I_{1,0}\cdot P_{m-1, \beta}\ +\ I_{2,0}\cdot P_{m-2,\beta}\ +\ \ldots\ .\]
\end{thm}

Here the $I_{k,0}=e(\Hilb^kX)$ are the Euler characteristics of the Hilbert schemes of 
points on $X$, and $Z^I_0(X)$ is their generating series. By \cite{Cheah} this is
$$
Z^I_0(X)=M(t)^{e(X)}.
$$

In fact we prove a little more. Fixing a Cohen-Macaulay $C$ in class $\beta$,
define $I_{n,C}$ to be the Euler characteristic of the subset of $I_n(X,\beta)$ consisting of subschemes whose underlying Cohen-Macaulay curve is $C$ (this is naturally a projective scheme, see below). Similarly let $P_{n,C}$
be the Euler characteristic of the subset (in fact projective scheme) of $P_n(X,\beta)$ of pairs supported on $C$.

\begin{thm} \label{EulerPT}
Let $C\subset X$ be a Cohen-Macaulay curve in a smooth projective threefold. Then
$$
I_{n,C}=P_{n,C}+e(X)P_{n-1,C}+e(\Hilb^2X)P_{n-2,C}+\ldots+e(\Hilb^nX)P_{0,C\,}.
$$
\end{thm}

We explain in Section \ref{gencase} how to deduce Theorem \ref{main} from
this by ``integrating" over the space of Cohen-Macaulay curves $C$. 

Naively one should think of the above identities as reflecting the decomposition of $I_n(X,\beta)$
into a union of the subset of pure Cohen-Macaulay curves with no free or embedded points, the subset with one free or embedded point, the subset with two points, etc. Birationally
such a decomposition is given by \eqref{J11} below.

\begin{rmk}
We emphasise that the identity of Euler characteristics in Theorem \ref{main} holds for \emph{any} threefold. (Only the $\chi^B$-weighted Euler characteristic motivation requires
the Calabi-Yau condition.) From a virtual class perspective it is very surprising that
the result should hold in the non-Calabi-Yau case, as we discuss at the end of Section
\ref{firstProof}. We also show there how to use Theorem \ref{EulerPT} to add incidence conditions, getting a topological analogue of the DT/PT conjecture with insertions for general threefolds.
\end{rmk}

An intermediate step in our proof is given in Section \ref{gitSection}, which may also be interesting in its own right. The main result there is Theorem \ref{gitwc}, which shows that the moduli spaces $I_n(X, \beta)$ and $P_n(X,\beta)$ are related by a special wall crossing in the sense of Geometric Invariant Theory. This result holds in any dimension. In fact a single wall is crossed, so we get natural morphisms $\varphi_I, \varphi_P$ to the moduli scheme of semistable objects modulo S-equivalence (pointwise this coincides
with the set of polystable objects) on the wall, which we call $SS_n(X, \beta)$,
\beq{J11} 
\hspace{-2cm} \xymatrix{
I_n(X,\beta) \ar[dr]_{\varphi_I}& & P_n(X,\beta) \ar[dl]^{\varphi_P}\\
& SS_n(X,\beta)\quad =\ \coprod^n_{k = 0} I^{pur}_{n-k}(X,\beta)\times S^k X.\hspace{-5cm}
}\eeq 
(The latter equality is a stratification by locally closed subschemes). The morphisms $\varphi_I, \varphi_P$ have a simple geometric meaning: they
separate the pure and torsion part of a subscheme (respectively the support and cokernel
of a stable pair). This produces a Cohen-Macaulay curve in $I^{pur}_{n-k}(X,\beta)$ (the
subset of the Hilbert scheme consisting of subschemes of pure dimension 1) and a finite
number of points with multiplicity. Compatibly with the general theory of GIT wall crossings
\cite{dolga, thad}, on passing from $I_n(X,\beta)$ to $P_n(X,\beta)$ each irreducible
component either disappears, appears, or undergoes a birational
transformation. (Since \cite{dolga, thad} only work with normal quotients
they only see this behaviour for \emph{connected} components.)

\begin{rmk} As we note at the end of Section \ref{firstProof}, our methods also
give a relation between the Euler
characteristics of single fibres of $\varphi_I,\varphi_P$. This is a punctual analogue of Theorem \ref{EulerPT}. Let $C$ be a Cohen-Macaulay curve in $\mathbb{A}^3$ and $p$ a closed point. Then one has
\beq{punctual}
\sum_{n \geq 0} e(\varphi^{-1}_I(C, n p)) t^n = M(t) \sum_{n \geq 0} e(\varphi^{-1}_P(C, n p)) t^n.
\eeq  
This holds whether $p$ lies on $C$ or not; in the latter case it is Cheah's formula for the punctual Hilbert schemes $\Hilb^n(p)$,
\[\sum_{n \geq 0} e(\Hilb^n(p))t^n = M(t).\]
\end{rmk}

The more fundamental ingredient in our proof is Joyce's motivic Ringel-Hall algebra \cite{JoyceII}. Theorem \ref{main} reflects identities between elements of this algebra naturally associated to the wall crossing morphisms $\varphi_I, \varphi_P$.

In fact it was immediately clear from \cite{PT1} that for Calabi-Yau 3-folds $X$, some kind of wall crossing formula for the Euler characteristics $I_{m,\beta},\,P_{m,\beta}$ should follow from the full weight of Joyce's theory (in particular \cite{JoyceIV}).
Working in the abelian subcategory of $D^b(X)$ given by tilting the usual one by the
subcategory of dimension zero sheaves one could then try to unravel the formulae in
\cite{JoyceIV} in this special case.

This is hard, however, and would use results from most of Joyce's intimidating
theory. This is because the wall crossing of \cite{PT1}, described in Section \ref{ptwc} below, relates ideal sheaves of 1-dimensional subschemes of $X$ to stable pairs and \emph{0-dimensional sheaves} (see for instance (\ref{ideal}, \ref{pair}) below). These 0-dimensional sheaves are strictly semistable, and have automorphisms. Joyce has a theory counting (in the sense of Euler characteristics, but yielding rational numbers) such things, and a wall crossing involving certain sums over graphs. Then one should relate this count to the count of ideal sheaves of 0-dimensional subschemes. Toda is not afraid of Joyce's work and has carried this out \cite{TodaPre}, as well as a related (and even harder) set of wall crossings in \cite{TodaPT}.

A quicker and apparently simpler proof was shown to us by Tom Bridgeland \cite{Br}, but using the full force of Kontsevich and Soibelman's theory \cite{KS}
(which is conjectural in parts). This meant that his proof was also something of a mystery to us. Nagao also found a proof for small crepant resolutions of affine toric Calabi-Yau 3-folds using all of Joyce's theory, and one using Kontsevich-Soibelman \cite{Nagao}.

Our original intention was to understand Bridgeland's work from first principles, reducing it to a down-to-earth analysis of the GIT wall crossing between moduli spaces derived in
Section 3. We soon discovered why that was impossible, but managed to find a way to make a
similar proof work (at the level of Euler characteristics) without using Kontsevich-Soibelman's
algebra homomorphism, making do instead with Joyce's (proved, and more easily understood)
virtual Poincar\'e polynomials on motivic Ringel-Hall algebras.
In particular, we do not
need to use Joyce's counting invariants for strictly semistable objects, nor his remarkable theory of virtual indecomposables.

So we decided to use this paper as an opportunity to give an elementary introduction to a small part of Joyce's work, and explain why it is necessary (and GIT is not sufficient) to handle wall crossing formulae. We illustrate this with a couple of proofs of the wall
crossing formula, one along the lines of Bridgeland's, the second in Section \ref{secondProof}
remarkably short and simple and closer to the philosophy of stability conditions.

Bridgeland has now found a way to carry out his proof using Joyce's work and the weighting
$\chi^B$, thus proving Conjecture \ref{conj} without assuming any other conjectures.
We decided
that this paper still had some value, however, in the GIT construction, the introduction
to Joyce's work, and the simpler proof that we give in Section \ref{secondProof} avoiding
Joyce's virtual indecomposables. Finally the identity of Euler characteristics clearly
has independent interest of its own, regardless of any virtual class motivation.

\noindent\textbf{Acknowledgements.} We mainly learnt about the work of Joyce and
Kontsevich-Soibelman from talking to Tom Bridgeland, from his lectures and his
preprint \cite{Br}. We also benefitted from stealing his work and that of Yukinobu Toda. 
We would like to thank Jim Bryan, Daniel Huybrechts, Dominic Joyce, Davesh Maulik,
Kentaro Nagao, Rahul Pandharipande, Bal\'azs Szendr\H oi and Yukinobu Toda for useful conversations. Special thanks are due to David Steinberg and an anonymous referee for carefully reading the manuscript and making many useful suggestions. 
The first author thanks Simon Donaldson and the Royal Society for support
during a visit to Imperial College London, and Gang Tian for a visit to BICMR, Beijing.

\section{Coherent systems and pairs}

Let $(X, \O_X(1))$ be a smooth complex projective scheme of dimension $D$ with a very ample line bundle. By a sheaf $F$ we always mean a coherent sheaf of modules over $\O_X$. Its \emph{Hilbert polynomial} with respect to $\O_X(1)$ is $P_F(m):= \chi(F(m))$. This is a polynomial in $m$ with degree $d = \dim(F) = \dim(\operatorname{Supp}(F))$, where $\operatorname{Supp}(F) \subset X$ is the closed subscheme defined by the kernel of the canonical section $\O_X \to \mathscr{E}nd(F)$. 

Let $r(F)$ denote the \emph{multiplicity} of $F$, that is $d!$ times the leading coefficient of the Hilbert polynomial $P_F$,
\[P_F(m) = r(F) \frac{m^d}{d!} + O(m^{d-1}).\]
Then the \emph{reduced Hilbert polynomial} of $F$ is simply 
\[p_F = \frac{P_F}{r}.\]
The essential reference for the rest of this section is Le Potier \cite{lepot}.
\begin{defn} A \emph{coherent system} on $X$ is a pair $(\Gamma, F)$, where $F \in Coh(X)$ and $\Gamma \subset H^0(F)$. A \emph{morphism}
$f\colon(\Gamma', F') \to (\Gamma, F)$ is a morphism of sheaves $F' \to F$ such that
$f_*(\Gamma') \subset \Gamma$. The \emph{dimension} $d$ of a coherent system $(\Gamma, F)$ is the dimension of $F$; i.e. $d = \dim\operatorname{Supp}(F)$, \emph{not} $\dim(\Gamma)$.  
\end{defn}
From the next section we will restrict to the special case $\dim(\Gamma) = 1$, $d = 1$.
\begin{defn} A coherent system $(\Gamma, F)$ with $h^0(F)\geq 1, \dim\Gamma = 1$, $d = 1$ is called simply a \emph{pair}. As
a slight abuse of notation we also denote this by $(F,s)$ or $[s\!:\O_X \to F]$, where $s \in H^0(F)$ is a nonvanishing section (defined up to a nonzero scalar multiple). 
\end{defn}
A coherent system $(\Gamma, F)$ does not have an intrinsically defined Hilbert polynomial. Roughly speaking we must prescribe a weight for $\Gamma$. Let $Q$ be a positive rational polynomial.
\begin{defn}The \emph{reduced Hilbert polynomial} of $(\Gamma, F)$ with respect to $Q$ is given by 
\[p^Q_{(\Gamma, F)} = \frac{\dim(\Gamma)}{r(F)}Q + p_F.\]
\end{defn}
Le Potier introduced Gieseker stability conditions for coherent systems. 
\begin{defn}
A coherent system $(\Gamma, F)$ is $Q$-\emph{semistable} if 
\begin{enumerate}
\item[$\bullet$] the sheaf $F$ is pure: it has no nontrivial subsheaves of
dimension $\leq d-1$,
\item[$\bullet$] for any subsheaf $F' \subset F$ with $0 < r(F') < r(F)$, setting $\Gamma' = \Gamma \cap H^0(F')$, one has the inequality of reduced Hilbert polynomials
\[p^Q_{(\Gamma', F')} \le p^Q_{(\Gamma, F)\,}.\]
\end{enumerate} 
It is $Q$-\emph{stable} if in addition the above inequality is strict.
\end{defn}

Let $P_X = \chi(\O_X(m))$ denote the Hilbert polynomial of $X$. 
\begin{defn} We say a coherent system is \emph{(semi)stable} if it is $P_X$-(semi)stable. Similarly a \emph{(semi)stable pair} is simply a $P_X$-(semi)stable pair. 
\end{defn}
One can prove that replacing $P_X$  by any rational, positive polynomial $Q$ with $\deg(Q) \geq D$ in the above definition gives the same (semi)stable objects. 

An important feature of a semistable (i.e. $P_X$-semistable) coherent system $(\Gamma, F)$ is that $\Gamma$ must generate $F$ generically.
This is proved by applying the stability inequality to the subsheaf $F'\subset
F$ given by the image of $\Gamma$ in $F$ to show that $r(F')=r(F)$; see \cite[Proposition 4.4]{lepot}. 

In the case of interest to us one can show that semistability coincides with stability which in turn corresponds to a simple geometric condition.
\begin{prop}[\cite{PT1} Lemma 1.3]\label{stablepairs} A pair $s\!: \O_X \to F$ is semistable if and only if it is stable, and this holds precisely when 
\begin{enumerate}
\item[$\bullet$] the $1$-dimensional sheaf $F$ is pure,
\item[$\bullet$] the section $s$ has $0$-dimensional cokernel. 
\end{enumerate}
In this case $\operatorname{Supp}(F)$ is a Cohen-Macaulay curve $C$.
\end{prop}

Let us now discuss the moduli problem. For any scheme $S$ let $\mathscr{F}$ be a flat family of sheaves on $X$ with Hilbert polynomial $P$ parameterised by $S$, i.e. a sheaf on $X \times S$ which is flat over $S$. Let $p, q$ denote projections to $X, S$ respectively. One is tempted to define a family of coherent systems parameterised by $S$ simply as a locally free subsheaf of $q_*\mathscr{F}$. However the natural map $q_* \mathscr{F}(s) \to H^0(\mathscr{F}_s)$ for a closed point $s \in S$ need not be an isomorphism. This is why the Serre dual notion is used instead. Recall that $D:=\dim X$.
\begin{defn} A \emph{flat family of coherent systems on $X$ parameterised by $S$} is a sheaf $\mathscr{F}$ on $X \times S$, flat over $S$, together with a locally free quotient of the sheaf $\eext^{D}_{q}(\mathscr{F}, p^*\omega_X)$.
\end{defn}
It is shown in \cite[Lemma 4.9]{lepot} that, for any closed point $s \in S$, 
\[\eext^{D}_{q}(\mathscr{F}, p^*\omega_X)(s) \cong \Ext^{D}(\mathscr{F}_s, \omega_X).\] 
We now introduce the usual functor of families,
\[S \mapsto \underline{\operatorname{Syst}}_{X, Q}(P)(S),\]
mapping a scheme $S$ to the set of isomorphism classes of flat families of $Q$-semistable
coherent systems on $X$ whose underlying sheaf has Hilbert polynomial $P$, parametrised by $S$. The discussion
in \cite[Section 4.3]{lepot} shows that this is a well
defined contravariant functor. 

\begin{thm}\cite[Theorem 4.12]{lepot}  There is a projective scheme $\operatorname{Syst}_{X, Q}(P)$ co-representing the functor $\underline{\operatorname{Syst}}_{X, Q}(P)$. Its closed points are in one-to-one correspondence with S-equivalence classes of $Q$-semistable
coherent systems.
\end{thm}
\begin{defn} Let $P_\beta$ be the Hilbert polynomial $P_\beta(m) = m\int_{\beta} c_1(\O_X(1))+n$. Le Potier's space of stable pairs $P_n(X,\beta)$ is the moduli space $\operatorname{Syst}_{X, P_X}(P_{\beta})$. 
\end{defn}
\begin{rmk} The moduli space $P_n(X,\beta)$ is fine. This is because stable pairs have no nontrivial automorphisms. See \cite[Section 2.3]{PT1} for more details.
\end{rmk}

In order to prove the GIT wall crossing in the next section we need to take a different point of view. Namely we assume the conclusion of Proposition \ref{stablepairs} as the \emph{definition} of stable pairs, and provide an ad hoc construction for their moduli space which slightly differs from Le Potier's. This is because we need to make full use of the special features of the $d = 1$ case.
\section{GIT wall crossing}\label{gitSection}
Throughout this section $P_\beta$ denotes the Hilbert polynomial
\[P_\beta(m) = r m + n, \]
where the multiplicity $r$ is given by $\int_{\beta}c_1(\O_X(1))$.

Let us clarify what we mean by a GIT wall crossing.
\begin{defn}\label{wallcross} A GIT wall crossing between two projective schemes $\M_0,\M_1$ is given by a projective scheme $\mathcal{N}$ with an action of a reductive algebraic group $G$, ample $\mathbb{Q}$-linearisations $\mathcal{L}_0, \mathcal{L}_1$ for the action of $G$ and  $t^* \in (0,1)\cap \mathbb{Q}$ such that, if $t \in [0,1]\cap \mathbb{Q}$ and $\mathcal{L}_t = (1-t)\mathcal{L}_0 + t \mathcal{L}_1$, then
\begin{enumerate}
\item[$\bullet$] for $0\leq t < t^*$, 
\[\mathcal{N}^{s}(\mathcal{L}_t) = \mathcal{N}^{ss}(\mathcal{L}_t) = \mathcal{N}^{ss}(\mathcal{L}_0)\]
and $\mathcal{N}/\hspace{-3pt}/_{\mathcal{L}_0}G \cong\M_0$;
\item[$\bullet$] for $t^* < t \leq 1$, 
\[\mathcal{N}^{s}(\mathcal{L}_t) = \mathcal{N}^{ss}(\mathcal{L}_t) = \mathcal{N}^{ss}(\mathcal{L}_1)\]
and $\mathcal{N}/\hspace{-3pt}/_{\mathcal{L}_1}G \cong\M_1$.
\end{enumerate}
\end{defn}
Clearly then $\mathcal{N}/\hspace{-3pt}/_{\mathcal{L}_t}G \cong\M_0$ for $0\leq t < t^*$, $\mathcal{N}/\hspace{-3pt}/_{\mathcal{L}_t}G \cong\M_1$ for $t^* < t \leq 1$.
 
\begin{rmk} Note that, unlike for example \cite{thad} Lemma 3.2, we require  semistability to coincide with stability away from $t^*$.
\end{rmk}

A GIT wall crossing gives the following well known diagram.
\begin{lem}In the situation of Definition \ref{wallcross} there are inclusions
\[
\xymatrix{
\mathcal{N}^{s}(\mathcal{L}_0) \ar@{_{(}->}[dr] & & \mathcal{N}^{s}(\mathcal{L}_1)\ar@{^{(}->}[dl]\\ & \mathcal{N}^{ss}(\mathcal{L}_{t^*})
}\]
inducing morphisms
\[\xymatrix{
\mathcal{N}/\hspace{-3pt}/_{\mathcal{L}_0} G \ar[dr]_{\varphi_0}& &
\mathcal{N}/\hspace{-3pt}/_{\mathcal{L}_1} G \ar[dl]^{\varphi_1}\\
& \mathcal{N}/\hspace{-3pt}/_{\mathcal{L}_{t^*}}G .\!
}\]
The fibre of $\varphi_{i}$ over $\varphi_i(G\cdot x)$ for $i = 0, 1$ is given by 
\[\varphi^{-1}_i(\varphi_i(G\cdot x)) = \{G \cdot y : \overline{G\cdot x} \cap \overline{G\cdot y} \neq \emptyset \emph{\,\,\,in\,\,\,} \mathcal{N}^{ss}(\mathcal{L}_{t^*})\}.\]
\end{lem}
\begin{proof} The first part is an application of the Hilbert-Mumford criterion. For the second part we only use the standard description of the quotient map $\mathcal{N} \to \mathcal{N}/\hspace{-3pt}/_{\mathcal{L}_{t^*}} G$ in terms of closures of semistable orbits, and the fact that $\mathcal{N}^{ss}(\mathcal{L}_{i}) = \mathcal{N}^s(\mathcal{L}_{i})$ for $i = 0, 1$.
\end{proof}

While there exist many trivial GIT wall crossings we are only interested in those that give interesting morphisms $\varphi_0, \varphi_1$. 

The MNOP/stable pairs GIT wall crossing can be stated as follows. For any subscheme
$Z \in I_n(X,\beta)$ let $Z^{pur}$ denote its maximal Cohen-Macaulay closed subscheme
(so $Z^{pur}$ is pure of dimension $1$, and $\O_Z\to\O_{Z^{pur}}$ has
finite kernel). Also let $I^{pur}_n(X,\beta) \subset I_n(X,\beta)$ be the locus of Cohen-Macaulay closed subschemes, and write $S^k X$ for the $k$th symmetric product. 
\begin{thm}\label{gitwc} There is a GIT wall crossing between the moduli spaces $I_n(X, \beta)$ and $P_n(X, \beta)$ as in Definition \ref{wallcross}, for which the following holds. 

Write
\beq{triangle} 
\xymatrix{
I_n(X,\beta) \ar[dr]_{\varphi_I}& & P_n(X,\beta) \ar[dl]^{\varphi_P}\\
& SS_n(X,\beta) 
}\eeq  
for the wall crossing morphisms to $SS_n(X, \beta):=
\mathcal{N}/\hspace{-3pt}/_{\mathcal{L}_{t^*}}G$.
Then there is a stratification by locally closed subschemes
\beq{strata} 
SS_n(X,\beta) = \coprod^n_{k = 0} I^{pur}_{n-k}(X,\beta)\times S^k X,
\eeq 
and on closed points $\varphi_I$ is given by
\[\varphi_I([Z]) = \Big(Z^{pur}, \sum_p \operatorname{len}(\I_{Z^{pur}}/\I_Z)_{p}\,p\Big).\]
Similarly, for $(F, s) \in P_n(X,\beta)$, 
\[\varphi_P(F, s) = \Big(\operatorname{Supp}(F), \sum_p \operatorname{len}(F/
\operatorname{im}(s))_p\,p\Big).\]
\end{thm}  
\begin{rmk} When $X$ is the (projective over an affine) threefold given by a small resolution of the threefold ordinary double point $\{xy = zw\} \subset \C^4$ the same result has been proved by Nagao and Nakajima \cite{nagnak} using moduli of perverse sheaves.
\end{rmk}

\begin{exa} \emph{Rigid curve.} Let $C \subset X$ be a smooth rigid curve in the class $\beta$.
Then the Hilbert scheme $I_{1 + \chi(\O_C)}(X,\beta)$ contains an irreducible component $I_1(X, C) \cong \operatorname{Bl}_C X$. Similarly the space of stable pairs $P_{1+\chi(\O_C)}(X, \beta)$ contains a component $P_1(X, C) \cong C$, and $SS_1(X,\beta)$ contains a component $SS_1(X, C) \cong X$. Restricting to these components the triangle \eqref{triangle} becomes 
\[
\spreaddiagramcolumns{-1pc}
\spreaddiagramrows{-1pc}
\xymatrix{
\operatorname{Bl}_C \ar[dr]_{\pi} X & & C \ar@{^{(}->}[dl]^{i}\\
& X
}
\]   
where $\pi$, $i$ are respectively the blow up and the inclusion. This may look odd to the reader familiar with \cite{dolga}, \cite{thad} as $i$ is not birational. This happens because in this case the space  $\mathcal{N}$ of Definition \ref{wallcross} is not connected. At least locally near $C$ it has two connected components which can be identified respectively with pairs with onto section and pairs with pure support, as we explain below. Each of the two GIT quotients selects only one component.

We will see that the relevant $\mathcal{N}$ for us contains two distinguished closed subschemes $\mathcal{N}_I,\,\mathcal N_P$, given by the closure of subschemes parametrising pairs with surjective section and with pure support, respectively. For simplicity assume that $C$ is the only curve in the class $\beta$. Then in our case $\mathcal{N}_I$ consists of structure sheaves of subschemes given by the union of $C$ plus a point (in $X\take C$, or an embedded point along $C$) plus the obvious section,
while $\mathcal{N}_P$ consists of degree 1 line bundles plus section defined by a point on $C$. Therefore their intersection is empty. (This is not the general case; when the
curve $C$ moves and singularities develop its genus can change, bubbling off points.)
\end{exa}

\begin{exa} \emph{Surfaces.} When $X$ is a surface the triangle \eqref{triangle} takes
a special form. Divisors on a smooth algebraic surface are Gorenstein curves, and stable pairs supported on a Gorenstein curve $C$ are in one-to-one correspondence with closed subschemes of $C$ by sending
a pair $\O_C \to F$ to the ideal sheaf $F^*\to\O_C$. For a proof see \cite[Proposition B.5]{PT3}.

For a given degree $\beta$ there exists a unique $g$ (given by the adjunction
formula $2g-2 = \beta\cdot(K_X + \beta)$) such that $I_{1-g}(X, \beta)$ is the moduli
space of \emph{divisors} in the class $\beta$, with a universal divisor $\mathscr{D}$.
Let $\Hilb^n(\mathscr{D}/I_{1-g}(X, \beta))$ denote the relative Hilbert scheme of
points on the fibres of $\mathscr D\to I_{1-g}(X, \beta)$. By the result mentioned above,
for any $n \geq 0$ there is a one-to-one correspondence 
\[P_{1-g + n}(X, \beta) \to \Hilb^n(\mathscr{D}/I_{1-g}(X, \beta))\] 
which is in fact an isomorphism of schemes (\cite[Proposition B.8]{PT3}).
The triangle \eqref{triangle} then becomes
\[
\xymatrix{
I_{1-g+n}(X, \beta)\ \cong\ I_{1-g}(X, \beta)\times\Hilb^n(X) \hspace{-4cm}
\ar[dr]_{\mu} & & \Hilb^n(\mathscr{D}/I_{1-g}(X, \beta)) \ar[dl]\\
& I_{1-g}(X,\beta) \times S^n X.
}
\]
Here we are decomposing the Hilbert scheme of curves $I_{1-g+n}(X, \beta)$ on a surface into a divisorial part $I_{1-g}(X, \beta)$ and a punctual part $\Hilb^n(X)$ -- see for
example \cite{fog}. Then $\mu\colon I_{1-g+n}(X, \beta) \to I_{1-g}(X,\beta) \times S^n
X$ is the Hilbert to Chow morphism $\Hilb^n(X)\to S^n X$ times by the identity on
$I_{1-g}(X, \beta)$.
\end{exa}

We start by constructing the objects of Theorem \ref{gitwc}. We build on \cite[Chapter 4]{lepot}, but taking advantage of the very special case $d = \deg(P_\beta) = 1$. In brief, the master space $\mathcal{N}$ is going to be a closed subscheme of a $\quot$ scheme times a projective space. Then a suitable choice of linearisation will pick out either stable pairs or structure sheaves of $1$-dimensional subschemes.

\subsection{$\mathcal{N}$ and $G$}
There is a general boundedness result for the family of isomorphism classes of $Q$-semistable coherent systems $(\Gamma, F)$ with prescribed Hilbert polynomial $P_F$, \cite[Theorem 4.11]{lepot}. In fact we will only use the following easier result. 
\begin{lem}\label{boundedness} The set of isomorphism classes of sheaves underlying isomorphism classes of stable pairs $[s\!: \O_X \to F]$ with $P_F = P_\beta$ is bounded.
\end{lem}
\begin{proof} The set of (the structure sheaves of) Cohen-Macaulay curves $C$ supporting stable pairs of Hilbert polynomial $P_\beta$ is bounded, since they all lie in the union of a finite number of Hilbert schemes of curves (since arithmetic genus is bounded on
curves in the class $\beta$). By the cohomological characterisation of boundedness (see for instance \cite[Lemma 1.7.6]{hl}) this gives a uniform $m$ such that 
\[H^i(X, \O_C(m - i)) = 0\]
for $i > 0$ (recall that the opposite also holds, i.e. a uniform bound on the regularity implies boundedness when the family of Hilbert polynomials is finite). Let the extension
\[0 \to \O_C \to F \to Q \to 0\]
define a sheaf $F$ underlying a stable pair. Since $H^i(Q(k)) = 0$ for $i > 0$ and all $k$, we get a surjection
\[H^i(\O_C(k)) \to H^i(F(k)) \to 0\]
for $i > 0$ and all $k$. Choosing $k = m - i$ gives $H^i(F(m-i)) = 0$ and
therefore a uniform bound on the regularity of $F$. As we recalled this implies boundedness for the family of all $F$s.
\end{proof}

Fix a positive integer $m$ and let $V$ be a fixed vector space of dimension $P_\beta(m)$. We form the rank $P_\beta(m)$ locally free sheaf on $X$
\[\mathcal{H} = V \otimes \O_X(-m).\]  
Consider the Quot scheme of quotients of $\mathcal{H}$ with Hilbert polynomial $P_\beta$,
\[\bq = \quot(\mathcal{H}, P_\beta).\]
By Lemma \ref{boundedness} we can choose $m \gg 0$ such that, for any stable pair $[s\!: \O_X \to F]$ with $P_F = P_\beta$, the sheaf $F(m)$ is globally generated and $H^i(F(m)) = 0$ for $i > 0$. Therefore the isomorphism class of $F$ corresponds to a closed point of $\bq$ for $m \gg 0$, and this is unique modulo the natural action of $\Sl(V)$ on $\bq$. In other words there is a surjective map
\[H^0(F(m))\otimes\O_X(-m)\to F \to 0\]
and we can pick an isomorphism $H^0(F(m))\cong V$. Indeed the group $G$ that occurs in Theorem \ref{gitwc} is $\Sl(V)$ for $m \gg 0$.

Let us now construct $\mathcal{N}$. Let $R_m$ denote the $m$-th graded piece of the graded ring of $X$, 
\[R_m = H^0(\O_X(m)).\]
\begin{lem} For any sheaf $F$ and any positive integer $m$ there is a canonical injection
\[F \subset F(m)\otimes R^*_m.\]
\end{lem}
\begin{proof} The morphism $F \to F(m)\otimes R^*_m$ in question is the element of 
\[\Hom(F, F(m)\otimes R^*_m)\cong \Hom(R_m, \Hom(F, F(m)))\]
given by tensoring $s \in R_m \cong \Hom(\O_X, \O_X(m))$ by $F$. This induces an injection $F \subset F(m)\otimes R^*_m$ because $\O_X(m)$ is globally generated.
\end{proof}

Thus there is an induced inclusion
\[H^0(F)\subset H^0(F(m)) \otimes R^*_m.\]
Let $\Gamma$ denote the subspace of $H^0(F)$ generated by $s$ (in this section we sometimes write $(\Gamma, F)$ for a pair $(F, s)$). By the above discussion $\Gamma$ 
can be regarded as a point of the projective space $\PP(V \otimes R^*_m)$ of lines in
$V \otimes R^*_m$ (uniquely up to the action of $\Sl(V)$).

The upshot of this is that isomorphism classes of pairs $[s\!:\O_X \to F]$ with $P_F = P_\beta$ are realised as particular geometric points of 
\[\bb = \PP(V \otimes R^*_m) \times \bq,\]
modulo the action of $\Sl(V)$ (and this includes structure sheaves of subschemes with their canonical section). 

We spell out the conditions that a closed point $(\Phi, F) \in \bb$ must satisfy in order to arise from a pair $(\Gamma, F)$ under the correspondence above:
\begin{enumerate}
\item[$\bullet$] the natural map $V \to H^0(F(m))$ induced by $V\otimes\O
_X(-m) \to F$ must be an isomorphism;
\item[$\bullet$] the subspace $\Phi$ of $V \otimes R^*_m$ must correspond to a subspace $\Gamma$ of $H^0(F)$;
\item[$\bullet$] the resulting pair $(\Gamma, F)$ must be stable.
\end{enumerate}
The scheme $\bb$ is too large for the second condition to arise from GIT stability. Instead our scheme $\mathcal{N}$ will be a suitable closed subscheme of $\bb$, as we now explain.

The product $X \times \bq$ (with projections $p, q$ to $X, \bq$ respectively) carries a universal quotient sheaf $\mathscr{F}$. Lemma \ref{boundedness} implies that for $m \gg 0$ all stable pairs with $P_F = P_\beta$ correspond to geometric points (unique up to the natural action of $\Sl(V)$) of the projective bundle
\[\bp := \Proj\oplus_kS^k\big(\eext^D_q(\mathscr{F}, p^*\omega_X)\big)\]
of rank one locally free quotients of $\eext^D_q(\mathscr{F}, p^*\omega_X)$ over $\bq$. 

The evaluation map $V\otimes\O_{X\times \bq}\to \mathscr{F}(m)$ induces $V\otimes\O_{\bq}\to
q_*(\mathscr{F}(m))$. By relative Serre duality for the morphism $q$, its
dual is a map
\[\eext^D_q(\mathscr{F}(m), p^*\omega_X) \to V^*\otimes \O_{\bq}.\]
Let $T \subset \bq$ denote the open subscheme where this morphism
is surjective; by duality this is precisely the open subscheme where the natural map $V \to H^0(F(m))$ is an isomorphism.
The surjective morphisms of sheaves over $T$
\[
\spreaddiagramrows{-.5pc}
\xymatrix{
& \hspace{-15mm}\eext^D_q(\mathscr{F}(m), p^*\omega_X)|_T\otimes R_m\hspace{-15mm}
\ar[dl]\ar[dr] &\\
V^*\otimes R_m\otimes\O_{T} & &\eext^D_q(\mathscr{F}, p^*\omega_X)|_T
}
\]
induce, for $m\gg0$, $\Sl(V)$-equivariant closed immersions over $T$
\beq{immersions}
\spreaddiagramrows{-.5pc}
\xymatrix{
\PP(V \otimes R^*_m) \times T \ar[dr]& & \bp|_T\ar[dl]\\
& \hspace{-15mm}\PP\big(q_*(\mathscr{F}(m)) \otimes R_m^*\big)|_T\,.
\hspace{-15mm}}
\eeq
By definition of $T$ the left arrow is in fact an isomorphism, so we regard $\bp|_T$ as a closed subscheme of $\PP(V \otimes R^*_m) \times T$, and we can rewrite the above conditions on a closed point $(\Phi, F) \in \PP(V \otimes R^*_m)\times \bq $ in a different form:
\begin{enumerate}
\item[$\bullet$] $F \in T$;
\item[$\bullet$] $(\Phi, F) \in \PP(V \otimes R^*_m) \times T$ lies in $\bp|_T$;
\item[$\bullet$] $(\Phi, F)$ is stable.
\end{enumerate}
We will see that these conditions do arise from GIT stability in the following space.
\begin{defn}\label{deff}
The space $\mathcal{N}$ is the scheme-theoretic closure of the locally closed subscheme
\[\bp|_T \subset \PP(V \otimes R^*_m) \times T \subset \bb.\]
\end{defn}
\begin{lem} For $m \gg 0$ closed subschemes $Z \subset X$ with Hilbert polynomial $P_Z = P_\beta$ also correspond to closed points of $\mathcal{N}$, unique up to the action of $\Sl(V)$.  
\end{lem}
\begin{proof} We only need to replace the set of isomorphism classes of stable pairs with $P_F = P_\beta$ with its union with the set of structure sheaves $\O_Z$ with their canonical section $1\!: \O_X \to \O_Z$ such that $P_Z = P_\beta$.
\end{proof}

The reason why we compactify $\bp|_T$ by embedding into $\bb$ and taking the closure rather than working directly with $\bp$ will be explained below.

\subsection{$\mathcal{L}_0$ and $\mathcal{L}_1$}
For $l \gg m$, the Quot scheme $\bq$ admits the familiar Grothendieck embedding 
\[\iota_l\!:\bq \into \Gr(V\otimes H^0(\O_X(l-m)), P_\beta(l))\]
into the Grassmannian of $P_\beta(l)$ dimensional quotients.
Pulling back the Pl\"ucker line bundle on $\Gr(V\otimes H^0(\O_X(l-m)), P_\beta(l))$ (the very ample generator of the Picard group) gives a very ample line bundle on $\bq$ which we denote by $\iota^*_l\O_{\Gr}(1)$. We also write $\O_{\PP}(1)$ for the hyperplane bundle of $\PP(V \otimes R^*_m)$.
\begin{defn} Let $c_0, c_1$ be rational numbers with 
\[0 < c_0 < \frac{1}{r} < c_1.\]
Then, for $i = 0, 1$ and $l \gg m$, define ample $\mathbb{Q}$-linearisations for the action of $\Sl(V)$ on $\mathcal{N}$ by 
\[\mathcal{L}_i := \O_{\PP}(l)\boxtimes\iota^*_l\O_{\Gr}(1)^{\otimes c_i}|_{\mathcal{N}}.\]
\end{defn}
\begin{rmk} The relative tautological bundle $\mathscr{U}$ on $\bp$ is relatively ample, so the linearisation $\mathscr{U}(l) \otimes \O_{\Gr}(1)^{\otimes c}$ is ample on $\bp$ for $c = c(l) \gg 0$. However for GIT we need to work with a fixed $c$, independent of $l$. This is why we needed to find a different compactification of $\bp|_T$ where ampleness of $\mathcal{L}_i$, $i = 0, 1$ is evident. This approach is due to Le Potier, \cite[Section 4.8]{lepot}.  
\end{rmk}

\subsection{Application of the Hilbert-Mumford criterion}
Since $\mathcal{N}$ is a projective scheme and the linearisations $\mathcal{L}_i$, $i = 0, 1$ are ample we can apply the Hilbert-Mumford criterion to find the (semi)stable points for the action. Any $1$-parameter subgroup $\lambda\!: \C^* \into \Sl(V)$ is uniquely determined by the decomposition into weight spaces $V = \bigoplus_{k \in \mathbb{Z}} V_{k}$. For any quotient 
\[\rho\!: V\otimes\O_X(-m) \to F \to 0\]
with $P_F = P_\beta$, define increasing filtrations of $V$, $F$ respectively by 
\begin{align*}
V_{\leq k} &= \sum_{j \leq k} V_j,\\
F_{\leq k} &= \rho(V_{\leq k}\otimes\O_X(-m)).
\end{align*}
For $\Phi \in \PP(V \otimes R^*_m)$ we set  
\[\Phi_{\leq k} = \Phi \cap (V_{\leq k} \otimes R^*_m)\]
and
\[\Phi_{k} = \Phi_{\leq k}/ \Phi_{\leq k-1}.\]
(Of course $\dim(\Phi_{\leq k})$ can only jump once!)
The quotient objects for the filtration of $F$ are
\[F_k = F_{\leq k} / F_{\leq k-1}\]
so we find onto maps
\[\rho_k\!: V_k \otimes\O_X(-m)\to F_k \to 0\]
and inclusions
\[\Phi_k \subset V_k \otimes R^*_m.\]
\begin{lem}\label{hm} Let $(\Phi, F)$ be a closed point of $\mathcal{N}$. Then
\[\lim_{t \to 0}\lambda(t)\cdot(\Phi, F) = (\oplus_k \Phi_k, \oplus_k F_k)\]
and the Hilbert-Mumford weight $\mu^{\mathcal{L}_i}((\Phi, F),\lambda)$ is given by the quantity
\[\frac{1}{\dim(V)}\sum_k \Big(\dim(V)\big(c_i P_{F_{\leq k}}(l)-\dim(\Phi_{\leq k})l\big) - \dim(V_{\leq k})\big(c_i P_F(l)-l\big)\Big).\]
\end{lem}
\begin{proof} The sheaf part of the statement is standard, see for instance Lemmas 4.4.3 and 4.4.4 of \cite{hl}. But the $\Phi$ part of the statement can be seen as a special case of this, regarding $\Phi_{k}$ as a subsheaf of the sheaf $V_k \otimes R^*_m$ over a point. We should only be careful about signs: when we rearrange as in \cite{hl} Lemma 4.4.4 to compute the Hilbert-Mumford weight, the sign of terms coming from $\Phi_{\leq k}$ is the \emph{opposite} of the sign of terms coming from $F_{\leq k}$, since we take ``subsheaves" (i.e. subspaces!) of the ``sheaf" $V \otimes R^*_m$, rather than quotients.
\end{proof}

\begin{rmk} By this Lemma the weight $\mu^{\mathcal{L}_i}((\Phi, F),\lambda)$ equals the constant $\dim(V)^{-1}$ multiplied by the sum of a finite number of quantities of the form
\begin{equation}\label{weight}
\mu^i(V') = \dim(V)(c_i P_{F'}(l)-\dim(\Phi')l) - \dim(V')(c_i P_F(l)-l),
\end{equation}
where $V' \subset V$, $F' \subset F$ is the subsheaf generated by $V'$, and 
\[\Phi' = \Phi\cap (V' \otimes R^*_m).\]
Conversely, for any subspace $V' \subset V$, we may define a one parameter subgroup $\lambda \in \Sl(V)$ by $\lambda(t) = t^{\dim(V')-\dim(V)}\operatorname{id}_{V'}\,\oplus\ t^{\dim(V')} \operatorname{id}_{V/V'}$, whose Hilbert-Mumford weight $\mu^{\mathcal{L}_i}((\Phi, F),\lambda)$ for $i = 0, 1$ is precisely the quantity $\mu^i(V')$ defined by \eqref{weight}. 
\end{rmk}
\begin{prop}\label{pairs_moduli} A stable pair $(F, s)$ with $P_F = P_\beta$ and a marking $H^0(F(m))\cong V$ corresponds to a closed point of $\mathcal{N}$ which is stable under the action of $\Sl(V)$ linearised on $\mathcal{L}_1$.

Conversely, the semistable locus $\mathcal{N}^{ss}(\mathcal{L}_1)$ coincides with the stable locus $\mathcal{N}^s(\mathcal{L}_1)$, and a closed point of $\mathcal{N}^s(\mathcal{L}_1)$ corresponds to a stable pair $(F, s)$ with $P_F = P_\beta$ and a marking of $H^0(F(m))$.

Therefore $\mathcal{N}^{s}(\mathcal{L}_1)$ has the good quotient $P_n(X,\beta)$ for the action of $\Sl(V)$.
\end{prop}
\begin{proof} Let $(\Gamma, F)$ be a stable pair. By Lemma \ref{hm} (and the remark after its proof) a closed point $(\Phi, F) \in \mathcal{N}$ belongs to $\mathcal{N}^{ss}(\mathcal{L}_1)$ if and only if, for every subspace $0 \subsetneq V' \subsetneq V$, we have
\beq{ss1}
\dim(V)(c_1 P_{F'}(l)-\dim(\Phi')l) \geq \dim(V')(c_1 P_F(l)-l).
\eeq
Here $F' \subset F$ denotes the subsheaf generated by $V'$. Replacing $\geq$ by $>$
gives the corresponding statement for $\mathcal{N}^{s}(\mathcal{L}_1)$.

When $(\Phi, F)=(\Gamma, F)$ represents a stable pair, then setting $\Gamma' = \Gamma \cap H^0(F')$ in the above we claim that in fact we always have GIT stability:
\[\dim(V)(c_1 P_{F'}(l)-\dim(\Gamma')l) > \dim(V')(c_1 P_F(l)-l).
\]
Since $(\Gamma, F)$ is a stable pair, we know a priori that $F$ is pure, so $V'$ generates a $1$-dimensional subsheaf with multiplicity $r' > 0$. By the choice of $c_1$ it is enough to prove the inequality of leading order terms
\[\dim(V)(c_1 r' - \dim(\Gamma')) > \dim(V')(c_1 r - 1).\]
If $\dim(\Gamma') = 0$ then we need to prove the inequality
\[\frac{\dim(V')}{\dim(V)} < \frac{r'}{r}\left(\frac{1}{1-\frac{1}{c_1 r}}\right).\]
But 
\[\frac{r \dim(V')}{r' \dim(V)} = \frac{1 + \frac{\chi(F')}{r'm}}{1 + \frac{\chi(F)}{r m}}\,,\]
and since there are only finitely many possible $r'$ for a fixed $r$, the term is converging to $1$ as $m \to \infty$ \emph{uniformly} over all $F' \subset F$, which implies the desired inequality for a uniform $m \gg 0$. If on the other hand $\dim(\Gamma') = \dim(\Gamma)$, that is $\operatorname{im}(s)\subset F'$, the assumption that $(\Gamma, F)$ is a stable pair says that $F/F'$ is $0$-dimensional and so $r = r'$. Using this information, the inequality to prove becomes
\[\dim(V) (c_1 r - 1)  > \dim(V')(c_1 r - 1),\]
which follows from $\dim(V') < \dim(V)$ since $c_1 r > 1$.

For the converse, let us first show that for $(\Phi, F) \in \mathcal{N}^{ss}(\mathcal{L}_1)$, the canonical map $V \to H^0(F(m))$ is an isomorphism, so indeed the quotient $V \otimes \O_X(-m) \to F \to 0$ corresponds to a sheaf $F$ plus a marking of $H^0(F(m))$. We only need to prove it is injective. Let $V'$ denote its kernel, so that the subsheaf $F'$ it generates is trivial. Substituting in \eqref{ss1} we get
\[-\dim(V)\dim(\Phi')l \geq \dim(V')((c_1 r - 1)l+c_1\chi(F)).
\]
By definition $c_1 > \frac{1}{r}$, so this is impossible unless $V' = \{0\}$, as required. In particular
\[\mathcal{N}^{ss}(\mathcal{L}_1) \subset \PP(V \otimes R^*_m) \times T.\] 
By the definition of $\mathcal{N}$, $(\Phi, F)$ is a point in the closure of $\bp|_T$ inside $\bb$, which moreover by the above discussion lies in $\PP(V \otimes R^*_m) \times T$. But $\bp|_T$ is a closed subscheme of $\PP(V \otimes R^*_m) \times T$ by the diagram of closed immersions \eqref{immersions}, so $(\Phi, F)$ really lies in $\bp|_T$, i.e. $\Phi$ is some subspace $\Gamma \subset H^0(F)$ generated by a nontrivial section $s$. From now on we will therefore write $(\Gamma, F)$ in place of $(\Phi, F)$. 

It remains to show that $(\Gamma, F)$ is a stable pair (so in particular, by the first part of the proof, $(\Gamma, F) \in \mathcal{N}^{s}(\mathcal{L}_1)$). To see that $F$ is pure, suppose $F' \subset F$ is a lower dimensional subsheaf. But then $F'$ must be $0$-dimensional and so generated by its global sections, i.e. $F'$ is generated by some proper subspace $V' \subset V$. In the inequality \eqref{ss1} we would get
\[\dim(V)(c_1 \chi(F') - \dim(\Gamma')l) \geq \dim(V')((c_1 r - 1)l + c_1\chi(F)),\]
where $\Gamma' = \Gamma \cap H^0(F')$. Since $c_1 > \frac{1}{r}$ this can never be satisfied. To see that $\Gamma$ generates $F$ generically, suppose that the section $s$ factors through $F' \subset F$ with multiplicity $r'$ ($r' > 0$ since we have shown that $F'$ must be $1$-dimensional). Then $\dim(\Gamma') = 1$ and in \eqref{ss1} we get
\[\dim(V)((c_1 r'- 1)l + c_1 \chi(F')) \geq \dim(V')((c_1 r - 1)l + c_1\chi(F)).\]
This implies
\[\dim(V)(c_1 r'- 1) \geq \dim(V')(c_1 r - 1),\]
which may be rewritten as
\[1 \leq \frac{r'}{r} - (c_1 r - 1)\left(\frac{\dim(V')}{\dim(V)} - \frac{r'}{r}\right).\]
Now
\[\left(\frac{\dim(V')}{\dim(V)} - \frac{r'}{r}\right) = \frac{r'}{r}\left(\frac{1}{1+\frac{\chi(F)}{r m}}-1\right) + \frac{\chi(F')}{r m}\frac{1}{1+\frac{\chi(F)}{rm}}\,.\]
Since $0 < r' \leq r$ and $\chi(F') \leq \chi(F)$ this term is $O(m^{-1})$ \emph{uniformly} over all subsheaves $F' \subset F$. Therefore
\[1 + O(m^{-1}) \leq \frac{r'}{r} \leq 1\]
holds uniformly over all subsheaves $F' \subset F$. Since, for a fixed $r$, there are only a finite number of possible $r'$, this implies $r' = r$ for $m \gg 0$. It follows that the quotient $F/F'$ is $0$-dimensional as required.
\end{proof}

This proof may be seen as a toy model of the Simpson-Le Potier estimates \cite[Section 4.5]{lepot}, modified to take advantage of the assumption $d = 1$.\\
\begin{prop}\label{hilb_moduli} A structure sheaf $\O_Z$ with $P_Z = P_\beta$, its canonical section $1\!:\O_X \to \O_Z \to 0$, and a marking of $H^0(\O_Z(m))\cong V$ corresponds to a closed point of $\mathcal{N}$ which is stable under the action of $\Sl(V)$ linearised on $\mathcal{L}_0$.

Conversely the semistable locus $\mathcal{N}^{ss}(\mathcal{L}_0)$ coincides with the stable locus $\mathcal{N}^s(\mathcal{L}_0)$, and a closed point of $\mathcal{N}^s(\mathcal{L}_0)$ corresponds to a structure sheaf $\O_Z$ with $P_Z = P_\beta$ and the canonical section $1\!:\O_X \to \O_Z \to 0$.

Therefore $\mathcal{N}^{s}(\mathcal{L}_0)$ has the good quotient $I_n(X, \beta)$ for the action of $\Sl(V)$.
\end{prop}  
\begin{proof} To a structure sheaf $\O_Z$ we associate the point $(\Gamma, \O_Z) \in \mathcal{N}$ which corresponds to the subspace $\Gamma$ generated by $1\!:\O_X \to \O_Z$ plus a marking for $H^0(\O_Z(m))$. We claim that for $V' \subset V$ generating $F' \subset \O_Z$ the inequality
\[\dim(V)(c_0 P_{F'}(l)-\dim(\Gamma')l) > \dim(V')(c_0 P_F(l)-l)\]
always holds, where $\Gamma' = \Gamma \cap H^0(\O_Z)$. As before this proves the first part of the Proposition. To prove the inequality, by the choice of $c_0$ it is enough to prove the inequality of leading order terms,
\[\dim(V)(\dim(\Gamma') - c_0 r') < \dim(V')(1 - c_0 r).\] 
Note that for a proper subsheaf $F' \subset F$ we cannot have $\dim(\Gamma') = 1$ as the canonical section is onto. But if $\dim(\Gamma') = 0$ the above inequality is always satisfied since by definition of $c_0$ one has $1 - c_0 r > 0$. 

For the converse, start with $(\Phi, F) \in \mathcal{N}^{ss}(\mathcal{L}_0)$. The inequality
\beq{Jineq}
\dim(V)(c_0 P_{F'}(l)-\dim(\Phi')l) \geq \dim(V')(c_0 P_F(l)-l)
\eeq
holds for any $V' \subset V$ generating $F' \subset F$ and $\Phi' = \Phi\cap(V'\otimes R^*_m)$ by the Hilbert-Mumford criterion. We must show that it implies that $(\Phi, F)$ corresponds to a structure sheaf with its canonical section and the choice of a marking (so in particular, by the first part of the proof, $(\Phi, F) \in \mathcal{N}^{s}(\mathcal{L}_0)$). Taking leading coefficients in \eqref{Jineq} we find
\[\dim(V)(\dim(\Phi') - c_0 r') \leq \dim(V')(1 - c_0 r).\]
As $\dim(V) > \dim(V')$ this can only hold if 
\[\dim(\Phi') - c_0 r' < 1 - c_0 r.\]
If $\dim(\Phi') = 1$ this would imply $r' > r$, which is absurd. On the other hand for $\dim(\Phi') = 0$ the inequality is always satisfied, since by the choice of $c_0$ we have $1-c_0 r > 0$. Now let 
\[\pi\!: \mathcal{N}^{ss}(\mathcal{L}_0) \to \mathcal{N}/\hspace{-3pt}/_{\mathcal{L}_0}\Sl(V)\]
denote the quotient map. The argument above shows that there is an isomorphism
\[\pi(\mathcal{N}^{ss}(\mathcal{L}_0)\cap\bp|_T) \cong I_n(X,\beta).\]
Therefore we obtain a locally closed immersion
\[I_n(X,\beta) \into \mathcal{N}/\hspace{-3pt}/_{\mathcal{L}_0}\Sl(V)\]
with dense image. Since both schemes are proper and separated this is an isomorphism.
\end{proof}

\subsection{The GIT wall crossing (Theorem \ref{gitwc}).}
Let $(\Phi, F)$ be any closed point of $\mathcal{N}$, and fix the ample $\mathbb{Q}$-linearisation $\mathcal{L}_t = (1-t)\mathcal{L}_0 + t \mathcal{L}_1$. To check whether $(\Phi, F)$ belongs to $\mathcal{N}^{ss}(\mathcal{L}_t)$ we will use the Hilbert-Mumford criterion. For any one parameter subgroup $\lambda\!: \C^* \into \Sl(V)$ we have
\[\mu^{\mathcal{L}_t}((\Phi, F), \lambda) = (1-t)\mu^{\mathcal{L}_0}((\Phi, F), \lambda) + t\mu^{\mathcal{L}_1}((\Phi, F), \lambda).\] 
Since we already know that $\mathcal{N}^{ss}(\mathcal{L}_i) = \mathcal{N}^{s}(\mathcal{L}_i)$ for $i = 0, 1$ we may assume that $(\Phi, F)$ is stable with respect to $\mathcal{L}_0$ and unstable with respect to $\mathcal{L}_1$, or vice virsa. In the first case say (the other is similar) fix a one parameter subgroup $\lambda$ for which
\begin{equation}\label{opposites}
\mu^{\mathcal{L}_0}((\Phi, F), \lambda)\cdot\mu^{\mathcal{L}_1}((\Phi, F), \lambda) < 0.
\end{equation}
Clearly, when \eqref{opposites} holds, convexity implies that for any choice of $(\Phi, F)$ and $\lambda$ there exists a unique critical value $t$, $0 < t < 1$, with $\mu^{\mathcal{L}_{t}}((\Phi, F), \lambda) = 0$. 
 
We claim that in fact there exists a unique critical value $t^*$, $0 < t^* < 1$, such that when \eqref{opposites} holds one has
\[\mu^{\mathcal{L}_{t^*}}((\Phi, F), \lambda) = 0\]
\emph{independently of $(\Phi, F)$ and $\lambda$}. This will complete the proof.

By the remark following the proof of Lemma \ref{hm} we may assume without loss of generality that $\mu^{\mathcal{L}_i}((\Phi, F),\lambda)$ has the special form $\mu^i(V')$, i.e.
\[\mu^{\mathcal{L}_i}((\Phi, F),\lambda) = \dim(V)(c_i P_{F'}(l)-\dim(\Phi')l) - \dim(V')(c_i P_F(l)-l)\]
for $i = 0, 1$. Reading through the proofs of Propositions \ref{pairs_moduli} and \ref{hilb_moduli} we see that the inequality \eqref{opposites} can only hold when 
\begin{enumerate}
\item[$\bullet$] $F'$ is $0$-dimensional and $\dim(\Phi') = 0$, or
\item[$\bullet$] $F'$ is $1$-dimensional, $r = r'$ and $\dim(\Phi') = 1$.
\end{enumerate} 
The only possible value for $t^*$ is
\[t^* = \left(1 - \frac{\mu^{1}(V')}{\mu^{0}(V')}\right)^{-1}\]
and we must show that this is independent of $V'$, and moreover that the two values coincide.

When $\dim(F') = \dim(\Phi') = 0$ we get
\begin{align*}
\frac{\mu^0(V')}{\mu^1(V')} &= \frac{\dim(V) c_0 \chi(F') - \dim(V')((c_0 r - 1)l + c_0 \chi(F))}{\dim(V)c_1\chi(F') - \dim(V')((c_1 r - 1)l + c_1 \chi(F))}\\
&= \frac{(\dim(V)-\chi(F))c_0 - (c_0 r - 1)l}{(\dim(V)-\chi(F))c_1 - (c_1 r - 1)l}\\
&= \frac{r m c_0 - (c_0 r - 1) l}{r m c_1 - (c_1 r-1)l}\,,
\end{align*}
where we have used $\dim(V') = \chi(F')$. The last expression is clearly independent of $V'$. When $\dim(F') = \dim(\Phi) = 1$ and $r = r'$ we get 
\begin{align*}
\frac{\mu^0(V')}{\mu^1(V')} &= \frac{\dim(V) ((c_0 r - 1)l + c_0 \chi(F')) - \dim(V')((c_0 r - 1)l + c_0 \chi(F))}{\dim(V)((c_1 r - 1)l + c_1\chi(F')) - \dim(V')((c_1 r - 1)l + c_1 \chi(F))}\\
&= \frac{(\dim(V)-\dim(V'))(c_0 r - 1)l + c_0(\dim(V)\chi(F')-\dim(V')\chi(F))}{(\dim(V)-\dim(V'))(c_1 r - 1)l + c_1(\dim(V)\chi(F')-\dim(V')\chi(F))}\\
&= \frac{c_0(\dim(V')-\chi(F'))-(c_0 r - 1)l}{c_1(\dim(V')-\chi(F'))-(c_1 r - 1)l}\\
&= \frac{r m c_0 - (c_0 r-1)l}{r m c_1 - (c_1 r -1 )l}\,,
\end{align*}
as required. \smallskip

Finally we prove the stratification \eqref{strata} for $SS_n(X, \beta)$. By the  argument above the closed points of $\mathcal{N}^{ss}(\mathcal{L}_{t^*})\setminus\mathcal{N}^{s}(\mathcal{L}_{t^*})$ are in one-to-one correspondence with pairs $(F, s)$ with $P_F = P_\beta$ such that either
\begin{enumerate}
\item[$\bullet$] $F$ is the structure sheaf of a subscheme $Z \subset X$ which is not pure, or
\item[$\bullet$] $F$ is pure, but the section $s$ has a nontrivial $0$-dimensional cokernel $Q$.
\end{enumerate}
By the general theory the closed points of $\mathcal{N}/\!/_{\mathcal{L}_{t^*}}\Sl(V)$ are in one-to-one correspondence with the equivalence classes of the set of sheaves as above under the equivalence relation induced by the intersection of the closures of $\Sl(V)$-orbits inside $\mathcal{N}^{ss}(\mathcal{L}_{t^*})$. 

Let $F$ be pure with Cohen-Macaulay support $C$ and 
\[0 \to \O_C \to F \to Q \to 0,\]
$Q \neq 0$. Any quotient $0 \to Q' \to Q \to \O_p \to 0$ onto the structure sheaf of a closed point $p$ gives a sequence
\[0 \to \O_C \to F' \to Q' \to 0,\]
where $F'$ is a subsheaf of $F$ with quotient $\O_p$.

Choosing $V' = H^0(F'(m))$ induces a one parameter subgroup 
\[\lambda(t) = t^{\dim(V')-\dim(V)}\operatorname{id}_{V'}\,\oplus\ t^{\dim(V')} \operatorname{id}_{V/V'}\]
in $\Sl(V)$ with
\[\lim_{t \to 0} \lambda(t)\cdot(F, s) = (\overline{s}\oplus 0, F'\oplus\O_p),\]
\[\mu^{\mathcal{L}_{t^*}}((F, s), \lambda) = 0.\]
Therefore $(\overline{s}\oplus 0, F'\oplus\O_p)$ lies in the closure of $\Sl(V)\cdot(F, s)$ inside $\mathcal{N}^{ss}(\mathcal{L}_{t^*})$, and $\chi(Q') = \chi(Q) - 1$. By induction the equivalence class of $(F, s)$ under the relation above has a unique representative of the form
\[(\overline{s}\oplus 0, \O_C\oplus_p \O^{\oplus \operatorname{len}(Q)_p}_p).\] 

The case when $F$ is a structure sheaf $\O_Z$ is similar: for any injection $\O_p \hookrightarrow \O_Z$, where $p$ is a closed point, choosing $V' = H^0(\O_p)$ induces $\lambda$ such that 
\[\lim_{t \to 0}\lambda(t)\cdot(1, \O_Z) = (\overline{1}\oplus 0, (\O_Z/\O_p)\oplus\O_p),\]
\[\mu^{\mathcal{L}_t}((1, \O_Z), \lambda) = 0.\] 
By induction on $\chi(\I_{Z}/\I_{Z^{pur}})$ we obtain a unique representative of the form
\[(\overline{1}\oplus 0, \O_{Z^{pur}}\oplus_p \O^{\operatorname{len}(\I_{Z}/\I_{Z^{pur}})_p}_p).\]

\section{Introduction to wall crossing in the derived category}\label{firstProof}
% , and the topological DT/PT wall crossing
Throughout this section we fix a Cohen-Macaulay curve $C$ in a threefold $X$ and work only with those ideal sheaves and stable pairs whose underlying curve is $C$. By this we mean that any $Z\in I_n(X,\beta)$ has an underlying Cohen-Macaulay curve $C\subset Z$ such that $\I_C/\I_Z$ is a 0-dimensional sheaf $T$ (where $T$ stands for ``torsion"). Similarly any $(F,s)\in P_n(X,\beta)$ has support on a Cohen-Macaulay curve $C$ and 0-dimensional cokernel, which we also call $T$.

So we define
$$
I_n(X,C)\subset I_{n+\chi(\O_C)}(X,\beta)
$$
to be the locus of ideal sheaves $\I_Z\subset\I_C$ such that $\I_C/\I_Z$ is a 0-dimensional sheaf of length $n$. By Theorem \ref{gitwc} this is the projective scheme $\varphi_I^{-1}(C, S^n X)$. Similarly
$$
P_n(X,C)\subset P_{n+\chi(\O_C)}(X,\beta)
$$
is the locus of stable pairs $(F,s)$ with scheme theoretic support on $C$ and cokernel of length $n$. By Theorem \ref{gitwc} this is the projective scheme $\varphi_P^{-1}(C, S^n C)$.

We denote their Euler characteristics by $I_{n,C}$ and $P_{n,C}$ respectively. (This
differs from their definition in the introduction by a shift in $n$ of $\chi(\O_C)$.) We prove identities between them which are universal in $C$, allowing us to further ``integrate" them over all $C$ weighted by the functions $I_{n,C}$ and $P_{n,C}$. This will yield identities between the Euler characteristics of the moduli spaces $I_m(X,\beta)$ and $P_m(X,\beta)$, giving the Euler characteristic versions of Conjecture \ref{conj}.

\subsection{The simplest identity}
The relationship between $I_{1,C}$ and $P_{1,C}$ is very easy to see. (In the terminology of Theorem \ref{gitwc} we are considering the fibres of $\varphi_I$ and $\varphi_P$ over a point of the stratum with $k=1$, i.e. we are considering ideal sheaves and pairs with one free or embedded point.)

For each point $x\in X$ consider the space of ideal sheaves $\I_Z$ in $I_1(X,C)$ such that $\I_C/\I_Z\cong\O_x$. It is $\PP(\Hom(\I_C,\O_x))$, with $\I_Z$ being the kernel of the homomorphism. The moduli space $I_1(X,C)$ is the union over $x\in X$ of these projective spaces, each with Euler characteristic
$$
h_x:=\hom(\I_C,\O_x).
$$
Similarly, since stable pairs $\O_X\to F$ with cokernel $\O_x$
correspond to extensions $0\to\O_C\to F\to\O_x\to0$, the moduli space
$P_1(X,C)$ is the union over all $x\in X$ of $\PP(\Ext^1(\O_x,\O_C))$. This has Euler characteristic
$$
e_x:=\ext^1(\O_x,\O_C).
$$
While $h_x,\,e_x$ jump as $x$ moves in $X$, their \emph{difference} is constant:
$$
h_x-e_x=1.
$$
(This is a Riemann-Roch formula combined with Serre duality; see Lemma \ref{SDRR} below.) Therefore ``integrating" over all $x\in X$ we find that
\beq{1}
I_{1,C}-P_{1,C}=e(X).
\eeq

\subsection{The general case} \label{gencase}
Since $P_0(X,C)$ is the single point $(\O_C,1)$, we can rewrite \eqref{1} as
$$
I_{1,C}=P_{1,C}+e(X)P_{0,C}.
$$
This is the $n=1$ case of Theorem \ref{EulerPT}. Below we will prove the general
case,
\beq{ePT}
I_{n,C}=P_{n,C}+e(X)P_{n-1,C}+e(\Hilb^2X)P_{n-2,C}+\ldots+e(\Hilb^nX)P_{0,C},
\eeq
relating the Euler characteristics of the fibres of $\varphi_I,\,\varphi_P$ of Theorem
\ref{gitwc} for $k\le n$.
 
We can ``integrate" the identities \eqref{ePT} over all $C$. By this we mean
we take weighted Euler characteristics (with either side of \eqref{ePT} as weighting function)
over the good open subsets $I^{pur}_{\chi(\O_C)}(X,\beta)$ of
the Hilbert schemes $I_{\chi(\O_C)}(X,\beta)$ consisting of only Cohen-Macaulay curves.
(Since the holomorphic Euler characteristic of the underlying Cohen-Macaulay curve $C$
can jump in flat families of subschemes $Z$, we have to do this over all $\bigcup_k
I^{pur}_k(X,\beta)$.)

Setting $m=\chi(\O_C)+n$, this gives Theorem \ref{main},
$$
I_{m,\beta}=P_{m,\beta}+e(X)P_{m-1,\beta}+e(\Hilb^2X)P_{m-2,\beta}+\ldots+e(\Hilb^mX)P_{0,\beta}.
$$

However \eqref{ePT} will turn out to be \emph{much} harder to prove than \eqref{1}, and will require more machinery than just an analysis of the fibres of $\varphi_I,\,\varphi_P$. We will work up to this machinery gently, giving a basic introduction to wall crossing.

\subsection{Stability and the simplest wall crossing}
The above formula \eqref{1} is the prototypical wall crossing formula, as we now explain.

We will be deliberately vague about stability conditions for a number of reasons. There are plenty of references for this topic, and anyway we do not strictly need them for more than motivation. The ideal scenario that we would like to work in has still not been realised: no one has yet managed to produce Bridgeland stability conditions on the derived category of coherent sheaves on a projective Calabi-Yau threefold. There are alternative fixes such as the variants of \cite{Bayer, TodaStab} or stability conditions on abelian categories.
The reader could just think of Gieseker- or slope- stability of vector bundles in what follows, or better (once we get back to Hilbert schemes and stable pairs) the GIT stability of the first half of this paper. \medskip

The model case that all wall crossings build on, and of which \eqref{1} is an analogue, is the following. \smallskip

Suppose that we have a moduli space of sheaves $E$ which are stable with respect to some slope function defined on the K-theory.  Suppose that we move the stability condition so that the phase of a subsheaf $A$ of $E$ has the same phase as $E$, and suppose further that the only classes in topological K-theory whose phases coincide are linear combinations of $[A]$ and $[E]$ (this is called a \emph{codimension one wall} in the space of stability conditions).

Then $E$ becomes strictly semistable due to exact sequences of the form
\beq{AB}
0\to A\to E\to B\to0.
\eeq
As our stability condition passes to the other side of the wall this exact sequence strictly destabilises $E$ and we lose it from the moduli space. However we can gain new stable sheaves -- nontrivial extensions $F$ in the opposite direction,
\beq{BA}
0\to B\to F\to A\to0.
\eeq
So long as we do not cross any other walls, and so long as $A$ and $B$ themselves remain stable on both sides of the wall (for instance if the K-theory classes of $A$ and $B$ are distinct and primitive then this is guaranteed) then $F$ is indeed stable and this is the only thing that can happen.

In this case elementary properties of stability show that $A$ and $B$ are simple ($\Aut=\C^*$) and have no homomorphisms between them. In the Calabi-Yau threefold case, Serre duality therefore forces the vanishing of all Ext groups between $A$ and $B$ except for
$$
\Ext^1(B,A)\cong\Ext^2(A,B)^* \qquad\text{and}\qquad \Ext^2(B,A)\cong\Ext^1(A,B)^*.
$$ 
The first group governs extensions \eqref{AB}, the second controls \eqref{BA}.

Therefore on passing through the wall we lose a $\PP(\Ext^1(B,A))$ of sheaves \eqref{AB} and gain a $\PP(\Ext^1(A,B))$ of sheaves \eqref{BA}. These have Euler characteristics $\ext^1(B,A)$ and $\ext^1(A,B)$ respectively. While we cannot control the jumping of these numbers as we vary $A$ and $B$, their \emph{difference} is the Mukai pairing of $A$ and $B,$ by Riemann-Roch and vanishing:
\begin{eqnarray*}
\ext^1(B,A)-\ext^1(A,B) \!\!&=&\!\! \hom(A,B)-\ext^1(A,B)+\ext^2(A,B)-\ext^3(A,B) \\ &=& \chi(A,B)\ =\ \int_Xch(A^\vee)ch(B)\operatorname{Td}(X).
\end{eqnarray*}
Therefore each $A$ and $B$ contribute to a change $\chi(A,B)$ in the Euler characteristic
of the moduli space of stable sheaves $E$.
Now ``integrate" this topological constant over the moduli spaces $\M_{[A]}$ and $\M_{[B]}$ of stable sheaves with the same topological K-theory classes as $A$ and $B$ respectively. The result is that on crossing the wall the Euler characteristic
of the moduli space of stable sheaves $E$ changes by
\beq{change}
\chi(A,B)e(\M_{[A]})e(\M_{[B]}).
\eeq

\subsection{Ideal sheaves and stable pairs} \label{ptwc}
As explained in \cite{PT1}, Conjecture \ref{conj} should be thought of as a similar wall crossing, but in the derived category. For what follows we will not need the Calabi-Yau
condition.

The exact sequences
\beq{ideal1}
0\to\I_Z\to\I_C\to T\to0
\eeq
and
\beq{pair1}
0\to\I_C\to\O_X\to F\to T\to0
\eeq
are equivalent in $D^b(X)$ to the exact triangles
\beq{ideal}
T[-1]\to\I_Z\to\I_C
\eeq
and
\beq{pair}
\I_C\to I\udot\to T[-1].
\eeq
(Recall that $I\udot$ is the complex $\{\O_X\to F\}\in D^b(X)$ with $\O_X$ in degree 0.)
These are reminiscent of \eqref{AB} and \eqref{BA}, but in the derived category instead of the abelian category of coherent sheaves, with the objects $T[-1]$ and $\I_C$ playing the roles of $A$ and $B$. 

\begin{lem} \label{SDRR}
On any threefold $X$, the only nonzero Ext groups between $T[-1]$ and $\I_C$ are
$$
\Ext^1(\I_C,T[-1]) \quad\text{and}\quad \Ext^1(T[-1],\I_C)
$$
and their Serre duals.

The first is $\Hom(\I_C,T)$; we denote its dimension by $h_T$. The second is isomorphic to $\Ext^1(T,\O_C)$, whose dimension we denote $e_T$. Then
\beq{RR}
h_T-e_T=\chi(\I_C,T)=[T],
\eeq
where $[T]$ is the length of the 0-dimensional sheaf $T$.
\end{lem}

\begin{proof}
Because $C$ is Cohen-Macaulay the sheaf $\I_C$ has homological dimension one. Therefore $\Ext^i(\I_C,T)$ vanishes for $i\ge2$. The only nonzero groups are $\Hom(\I_C,T)$ (the first of the groups above) and $\Ext^1(\I_C,T)$ (the Serre dual of the second of the groups above). This establishes the first claim, and simplifies Riemann-Roch to
$$
\chi(\I_C,T)=\hom(\I_C,T)-\ext^1(\I_C,T).
$$
Since this is a deformation invariant, we can compute it for $T$'s support disjoint from $C$. In this case we get the same answer by replacing $\I_C$ by $\O_X$, which yields
$$
\chi(\I_C,T)=\hom(\O_X,T)=[T].
$$

Finally the exact sequence $0\to\I_C\to\O_X\to\O_C\to0$ gives
$$
\Ext^1(T,\O_C)\cong\Ext^2(T,\I_C)=\Ext^1(T[-1],\I_C).
$$
This is Serre dual to $\Ext^1(\I_C,T\otimes\omega_X)\cong\Ext^1(\I_C,T)$, so $e_T=\ext^1(\I_C,T)$. Therefore $\chi(\I_C,T)=h_T-e_T$ as claimed.
\end{proof}

This result is compatible with there being a stability condition for which the stable objects of Chern character $(1,0,-\beta,-n+\beta.c_1(X)/2)$ are the ideal sheaves of $I_n(X,\beta)$ on one side of a codimension one wall and the stable pairs of $P_n(X,\beta)$ on the other side of the wall. In fact this can be justified using the variants of stability conditions
defined by Bayer \cite{Bayer} and Toda \cite{TodaStab}, but we shall just use that (and the GIT stabilities of the first half of this paper) as motivation.

The first of the above groups, $\Ext^1(\I_C,T[-1])$, governs the extensions $\I_Z$ in \eqref{ideal}. Via its isomorphism to $\Hom(\I_C,T)$ we see the extension $\I_Z$ as the kernel of a surjection $\I_C\to T$ as in \eqref{ideal1}. Similarly the second group $\Ext^1(\I_C,T[-1])$ governs the extensions $I\udot$ in \eqref{pair}. Via its isomorphism to $\Ext^1(T,\O_C)$
we see this extension $I\udot$ as entirely equivalent to the extension $0\to\O_C\to F\to T\to0$. \medskip

When $T$ has length 1 we recover the simplest type of wall crossing: $A=T[-1]$ and $B=\I_C$ have distinct primitive K-theory classes and so are stable on both sides of the wall -- they do not break up into further pieces. So \eqref{change} applies to give the change in invariants as we cross the wall.

Since $[T]=1$, the sheaf $T$ is necessarily the structure sheaf $\O_x$ of a point $x\in X$. Therefore $\M_{[A]}\cong X$ in \eqref{change}, while $\M_{[B]}$ is a single point because we have fixed the curve $C$. Thus \eqref{change} becomes precisely \eqref{1}:
$$
I_{1,C}-P_{1,C}=e(X).
$$
Of course putting $A=T[-1]$ and $B=\I_C$ into the argument that lead to \eqref{change} we get precisely the argument we gave to prove \eqref{1}.

\subsection{The two point case} \label{2pt}
As soon as $[T]=2$ the wall crossing becomes more complicated. The sheaves $T$ will only be semistable, so they split further into stable pieces, and their automorphism groups are bigger.

Given any 0-dimensional sheaf $T$ (of any length), the ideal sheaves $\I_Z$ such that $\I_C/\I_Z\cong T$ correspond to the surjections $\I_C\to T\to0$ modulo the automorphisms of $T$:
\beq{onto}
\frac{\Onto(\I_C,T)}{\Aut(T)}\,.
\eeq
Similarly the stable pairs supported on $C$ whose cokernel is isomorphic to $T$ form the space
\beq{pure}
\frac{\Pure(T,\O_C)}{\Aut(T)}\,,
\eeq
where $\Pure(T,\O_C)\subset\Ext^1(T,\O_C)$ is the subset of those extensions $0\to\O_C\to F\to T\to0$ for which the extension $F$ is pure. Just as the surjections \eqref{onto} correspond to homomorphisms that do not factor through any proper subsheaf, it is easy to see that the pure extensions are those which do not factor through any proper quotient $T\to Q\to0$ of $T$ via the map $\Ext^1(Q,\O_C)\to\Ext^1(T,\O_C)$.

We apply this to the three possibilities for $T$ of length 2.

\begin{itemize}
\item[(a)] $T=\O_x\oplus\O_y$, where $x\ne y\in X$.

Setting $H_x:=\Hom(\I_C,\O_x)$ and $E_x:=\Ext^1(\O_x,\O_C)$, \eqref{onto} becomes
\beq{xy}
\Big(\big(H_x\oplus H_y\ \big)\ \take\ \big(H_x \ \cup\ H_y \big)\Big)\Big/(\C^*\times\C^*).
\eeq
Of course this is just $\PP(H_x)\times\PP(H_y)$ with Euler characteristic $h_xh_y$. Similarly the pure extensions \eqref{pure} work out to be $\PP(E_x)
\times\PP(E_y)$ with Euler characteristic $e_xe_y$. Since $h_x-e_x=1=h_y-e_y$ we find the difference in Euler characteristics is
\beq{xy1}
h_xh_y-e_xe_y=e_x+e_y+1.
\eeq

\item[(b)] $T=\O_x\oplus\O_x$. In this case it is helpful to write $\Hom(\I_C,\O_x^{\oplus2})$ as $\Hom(\C^2,H_x)$, so that the surjections in the former group correspond to the rank 2 maps in the latter. Therefore \eqref{onto} becomes 
\beq{xx}
\Big(\Hom(\C^2,H_x)\ \take\ \bigcup_{\{\C \subset \C^2\}}\Hom(\C,H_x)
\Big)\Big/\operatorname{GL}(2,\C),
\eeq
i.e. the Grassmannian $\Gr(2,H_x)$ of 2-dimensional subspaces of $H_x$. This has Euler characteristic $h_x(h_x-1)/2$. In a parallel manner the pure extensions \eqref{pure} work out to be $\Gr(2,E_x)$ with Euler characteristic $e_x(e_x-1)/2$. The difference of these two is
\beq{xxex}
\left(\begin{array}{c} \!\!h_x\!\! \\ \!\!2\!\!\end{array}\right)
-\left(\begin{array}{c} \!\!e_x\!\! \\ \!\!2\!\!\end{array}\right)=e_x.
\eeq
\item[(c)] Finally let $T=\O_{2x}$ denote the structure sheaf of a length two subscheme of $X$ supported at $x$. (There is a $\PP^2$ of these.)

Let $\m_x$ be the maximal ideal of $\O_{2x}$ with quotient $\O_x$. Then the surjections \eqref{onto} are those homomorphisms which do not factor through $\m_x$, i.e.
\beq{2x}
\big(\Hom(\I_C,\O_{2x})\ \take\ \Hom(\I_C,\m_x)\big)\Big/(\C^*\ltimes\C),
\eeq
which is
$$
\PP(\Hom(\I_C,\O_{2x}))\ \take\ \PP(\Hom(\I_C,\m_x))\big/\C.
$$
Since $e(\C)=1$ this has Euler characteristic $h_{2x}-h_{\m_x}$, in the obvious notation.
The pure extensions are those which do not factor through the quotient $\O_{2x}\to\O_x$, i.e. \eqref{pure} becomes
$$
\big(\Hom(\O_{2x},\O_C)\ \take\ \Hom(\O_x,\O_C)\big)\Big/(\C^*\times\C),
$$
with Euler characteristic $e_{2x}-e_x$. The difference between these two Euler characteristics is
\beq{2x1}
\qquad\qquad (h_{2x}-h_{\m_x})-(e_{2x}-e_x)=(h_{2x}-e_{2x})-(h_{\m_x}-e_x)=2-1=1,
\eeq
using the apparent coincidence that $\m_x\cong\O_x$ as abstract sheaves. (We will return to this point.)
\end{itemize}

We now ``integrate" these differences in Euler characteristics over all length two sheaves $T$ to calculate $I_{2,C}-P_{2,C}$.

The sheaves $T$ of cases (a) and (c) together cover $\Hilb^2(X)$; integrating the 1 which appears in \eqref{xy1} and \eqref{2x1} therefore gives $e(\Hilb^2(X))$.

From \eqref{xy1} this leaves the integral of $e_x+e_y$ over the space of $(\O_x\oplus\O_y)$s ($x\ne y$). Passing to the double cover (i.e. ordering $x$ and $y$) gives the integral of $\frac12(e_x+e_y)$ over the complement of the diagonal in $X\times X$. By symmetry this is the integral of $e_x$. Adding in the contribution \eqref{xxex} of case (b), i.e. the integral of $e_x$ over the diagonal, we get the weighted Euler characteristic of $X\times X$ with weight $e_x$ pulled back from the first factor.

Integrating over the first factor of the product gives $P_{1,C}$, then integrating over the second makes this $e(X)P_{1,C}$. The final formula is therefore what we wanted:
$$
I_{2,C}-P_{2,C}=e(X)P_{1,C}+e(\Hilb^2X).
$$

\subsection{The theories of Joyce and Kontsevich-Soibelman}
When we started out on this project we hoped to derive the wall crossing formula by an elementary analysis of the fibres of the morphisms $\varphi_I\,,\varphi_P$ of \eqref{J11}. We were not convinced of the necessity of stacks, or of Joyce's rather complicated theory (not to mention the even more complicated, partly conjectural theory of Kontsevich-Soibelman).

The reader not yet convinced that a new idea is needed is invited to continue the above analysis for 3 points. It turns out that there is a good reason that this naive analysis is not sufficient to handle the general case. There is a clever reordering of the sums involved that exposes a certain symmetry that underpins the wall crossing and which has been invisible to us up until now. This leads naturally into Hall algebras and (a small part of) the theories of Joyce and Kontsevich-Soibelman.

The reader happy with those works can safely jump to the proof(s) of the wall crossing formula. For all others (and there is currently no friendly reference) this section should serve as both motivation and an introduction.

Even in the length two case we found spaces, such as $\Gr(2,H_x)$, with Euler characteristic $h_x(h_x-1)/2$ which is nonlinear in $h_x$. Therefore its interaction with Serre duality and the (linear!) Riemann-Roch formula
was more complicated. We also got terms like $\PP(H_x)
\times\PP(H_y)$ whose Euler characteristic $h_x.h_y$ is not even a function of $\hom(\I_C,\O_x\oplus\O_y)=h_x+h_y$.

By \eqref{onto} these spaces all arise in the following way.
\begin{itemize}
\item[(i)] Start with $\Hom(\I_C,T)$,
\item[(ii)] remove the locus of non-onto Homs, and
\item[(iii)] divide by the automorphisms of $T$.
\end{itemize}
We have good control over (i) by Serre duality and Riemann-Roch:
$\Hom(\I_C,T)\cong\C^{h_T}$ has exponent linear in $h_T$, while on the pairs side the relevant space is $\Ext^1(T,\O_C)\cong\C^{e_T}$, and $h_T-e_T=[T]$ is topological -- the length of the 0-dimensional sheaf $T$.
The removal of the bad locus (ii) will be discussed in Section \ref{IncExc} below. Finally the quotienting (iii) suggests two things. Firstly, we should be using stacks, by which we just mean that we should remember the automorphism group of each point $T$ in the space of sheaves. And secondly, we should use virtual Poincar\'e polynomials instead of Euler characteristics: we cannot compute the Euler characteristic $e(X/G)$ of a free quotient as $e(X)/e(G)$ if both $X$ and $G$ have zero Euler characteristic, but we \emph{can} do this with virtual Poincar\'e polynomials.

\subsection{Virtual Poincar\'e and Serre polynomials}
Virtual Poincar\'e polynomials are defined for all quasi-projective varieties. They are motivic, satisfying $P_t(X\take Y)+P_t(Y)=P_t(X)$ and $P_t(X\times Y)=P_t(X)P_t(Y)$. Therefore they are determined by their values on smooth projective varieties, for which they equal the classical Poincar\'e polynomial $P_t(X)=\sum_i(-1)^ib_i(X)t^i$. (Deligne showed that they are then well defined.) In particular they satisfy $P_t(X/G)=P_t(X)/P_t(G)$ when $X\to X/G$ is a Zariski-locally trivial $G$-bundle, and $\lim_{t\to1}P_t(X)=e(X)$.

The spaces that interest us will be unions of affine spaces, whose mixed Hodge structure
is of level 0, so that $P_t(X)$ is in fact a polynomial in $q=t^2$. (In this case this
function of $q$ is the Serre polynomial, counting points over $\mathbb F_q$.)
We therefore replace the variable $t$ (which we will recycle for a different purpose below) by $q^{1/2}$ and abuse notation by denoting the virtual Poincar\'e
polynomial by $P_q(X)\in\Z[q^{1/2}]$ and calling it the Serre polynomial throughout.

The Serre polynomial of the affine line $\C$ is $q=t^2$, so that of $\Hom(\I_C,T)$ is $q^{h_T}$. This linearity in the exponent, over which we have the usual control by Serre duality and Riemann-Roch, will be crucial. When $[T]=1$ the space of interest to us is
$\PP(H_x)$, whose Serre polynomial is computed via the prescription \eqref{onto} as
$$
\frac{P_q(H_x\take\{0\})}{P_q(\C^*)}=\frac{q^{h_x}-1}{q-1}
=q^{h_x-1}+q^{h_x-2}+\ldots+q+1,
$$
the right hand side reflecting its natural cell structure (or its usual Poincar\'e polynomial). Taking the limit as $q\to1$ we recover $e(\PP(H_x))=h_x$, essentially as $\frac d{dq}\big|_{q=1}\ q^{h_T}$. This was all we required for the 1 point case.

For 2 points we had the 3 cases (a), (b) and (c). In case (a) the Serre polynomial of \eqref{xy} is
\beq{aSerre}
\frac{q^{h_x+h_y}-q^{h_x}-q^{h_y}+1}{(q-1)^2}=(q^{h_x-1}+\ldots+q+1)(q^{h_y-1}+\ldots+q+1),
\eeq
i.e. the Serre polynomial of $\PP(H_x)\times\PP(H_y)$, with $\lim_{q\to1}=h_x.h_y$.

In case (b) we take the Serre polynomial of \eqref{xx}. Removing the origin in each of
the vector spaces in \eqref{xx} so that the union becomes disjoint, we get
\beq{bSerre}
\frac{(q^{2h_x}-1)-(q+1)(q^{h_x}-1)}{(q^2-1)(q^2-q)}\ =\ \frac{(q^{h_x-1}+\ldots+q+1)
(q^{h_x-2}+\ldots+q+1)}{q+1}\,.
\eeq
Tending $q\to1$ gives $e(\Gr(2,H_x))=h_x(h_x-1)/2$.

In case (c) we compute the Serre polynomial of \eqref{2x} to be
\beq{cSerre}
(q^{h_{2x}}-q^{h_{\m_x}})\big/q(q-1)\ =\ 
q^{h_{2x}-2}+\ldots+q^{h_{\m_x}-1},
\eeq
with limit $h_{2x}-h_{\m_x}$ as $q\to1$.

So all the nonlinearity in the Euler characteristics is coming from taking the $q\to1$ limit, in which the automorphisms contribute denominators like $q-1$ (and its $n$th powers) which effectively differentiate ($n$ times) the nice controllable terms like $q^{h_x}$ which have linear exponents. In the next section we will avoid such problems by taking $\lim_{q\to1}$ \emph{after} applying Riemann-Roch and Serre duality (in contrast to what we just did).

\subsection{Inclusion-exclusion and reordering the sum} \label{IncExc}
More importantly, to deal with problem (ii) above and get the calculations to work out in more complicated examples, we need to reorder the relevant sum in a crucial way.

By the inclusion-exclusion principle, write $\Onto(\I_C,T)$ as 
\beq{incexc}
\ \Hom(\I_C,T)\ -\bigcup_{\{T_1<T\}}\Hom(\I_C,T_1)\ +
\bigcup_{\{T_1<T_2<T\}}\Hom(\I_C,T_1)\ -\ \ldots\,,
\eeq
where each $<$ denotes the inclusion of a subsheaf which may be zero but may not be the whole sheaf. We want to take the Serre polynomial of this, divided by that of $\Aut(T)$, and then ``add up" over all $T$ and take $q\to1$.
(Throughout this section, ``add up", ``sum" or ``integrate" are all meant in the sense of taking Serre polynomials.)

Think of subsheaves $T_1<T$ as extensions $0\to T_1\to T\to Q_1\to0$. Then, for instance, the sum \emph{over all $T$} of the second term in \eqref{incexc} can be reordered by instead summing over \emph{all} extensions between \emph{all} 0-dimensional sheaves $T_1,Q_1$. (We can do something similar for all of the terms of \eqref{incexc}.)

In this way we will see a new symmetry by applying Serre duality and Riemann-Roch to the extensions between $T_1$ and $Q_1$. Putting $T_1$ and $Q_1$ on an equal footing (they are both arbitrary $0$-dimensional sheaves) also resolves the asymmetry between $\m_x$ and $\O_x$ that we noted in the 2 point case. We now see both as $\O_x$, and the sheaves $\O_{2x}$ and $\O_x\oplus\O_x$ as extensions between $\O_x$ and itself. \medskip

As an example we spell out explicitly how this reorders the sum of the terms (a), (b), (c) in the 2-point case of Section \ref{2pt}. We concentrate on the second term in the inclusion-exclusion \eqref{incexc} (the first and third are easier in this case).

So we sum $\Hom(\I_C,\O_x)$ over equivalence classes of exact sequences $0\to\O_x\to T\to\O_y \to 0$ for all $x, y\in X$. Here exact sequences are equivalent if they differ by the action of some $g\in\Aut T$ (changing the inclusion by post-multiplication by $g$, and the quotient by pre-multiplication by $g^{-1}$). Under this $\Aut(T)$-action the stabiliser of an exact sequence is $\Hom(\O_y,\O_x)\subset\Aut(T)$, where $\Phi\in\Hom(\O_y,\O_x)$ induces the automorphism of $T$ given by $\id_T$ plus the composition $T\to\O_y\Rt{\Phi}\O_x\to T$. Therefore when we divide the exact sequences by $\Aut(T)$ we get the equivalence classes (corresponding to points in $\Ext^1(\O_y,\O_x)$) with a residual action of $\Hom(\O_y,\O_x)$ (acting trivially on $\Ext^1(\O_y,\O_x)$, though not on the corresponding extension $T$). So for each $x,y\in X$ this sum contributes the motive
\beq{motive}
-\Hom(\I_C,\O_x)\times\frac{\Ext^1(\O_y,\O_x)}{(\Aut(\O_y)\times\Aut(\O_x))\times\Hom(\O_y,\O_x)}
\eeq
to the second term in the inclusion-exclusion \eqref{incexc}. As usual we are thinking of the sheaves $\O_y,\O_x$ as belonging to the \emph{stack} of 0-dimensional sheaves, i.e. we are remembering their automorphism groups which act on the extensions between them. The contribution at the level of Serre polynomials is
\beq{serre}
-q^{h_x}\ \frac{q^{\ext^1(\O_y,\O_x)}}{q^{\hom(\O_y,\O_x)}(q-1)^2}\,.
\eeq
We have seen previously how Serre duality and Riemann-Roch applied to the first exponent $h_x$ are important in comparing with the corresponding $\Pure$ contributions to the Serre polynomial of the space of stable pairs. But now we can also apply Serre duality and Riemann-Roch to the other
terms in \eqref{serre} as well:
$$
\ext^1(\O_y,\O_x)-\hom(\O_y,\O_x)=\ext^1(\O_x,\O_y)-\hom(\O_x,\O_y),
$$
so that
$$
\frac{q^{\ext^1(\O_y,\O_x)}}{q^{\hom(\O_y,\O_x)}}=\frac{q^{\ext^1(\O_x,\O_y)}}
{q^{\hom(\O_x,\O_y)}}\,.
$$
This symmetry was invisible before we reordered the sum, but will be important later.

In analysing \eqref{motive} we treat the zero and nonzero extensions separately, thus
splitting \eqref{serre} into the sum
\beq{serre2}
-q^{h_x}\,\frac1{q^{\hom(\O_y,\O_x)}(q-1)^2}\ -\ 
q^{h_x}\,\frac{q^{\ext^1(\O_y,\O_x)}-1}{q^{\hom(\O_y,\O_x)}(q-1)^2}\,.
\eeq
We begin with the first term, corresponding to the zero extension $0\in\Ext^1(\O_y,\O_x)$. Thus the sheaf $T$ is $\O_x\oplus\O_y$ and we are in case (a) of Section \ref{2pt} if $x\ne y$ and case (b) if $x=y$. When
$x\ne y$, this extension contributes
$$
-q^{h_x}\,\frac1{(q-1)^2}
$$
to \eqref{serre2}, and indeed this appears in the Serre polynomial of (a) \eqref{xy} as the second term in the numerator of \eqref{aSerre}.
Similarly when $x=y$ it contributes
$$
-q^{h_x}\,\frac1{q(q-1)^2}
$$
to (b) \eqref{xx}. And rewriting this as
$$
-\frac{(q+1)q^{h_x}}{(q^2-1)(q^2-q)}
$$
we see that it appears on the left hand side of \eqref{bSerre}.

This leaves the nonzero elements of $\Ext^1(\O_y,\O_x)$. For $y\ne x$ there are none; for $y=x$ they contribute the rest of \eqref{serre2},
\beq{mess}
-q^{h_x}\,\frac{q^{\ext^1(\O_y,\O_x)}-1}{q^{\hom(\O_y,\O_x)}(q-1)^2}
\ =\ -q^{h_x}\,\frac{q^3-1}{q(q-1)^2}\ =\ -q^{h_x}\,\frac{q^2+q+1}{q(q-1)}\,,
\eeq
to the inclusion-exclusion description of $\Onto(\I_C,\O_{2x})$ of case (c) \eqref{2x}.
And indeed multiplying the negative term in \eqref{cSerre} by $(q^2+q+1)$ (the Serre
polynomial of the space $\PP^2$ of sheaves $\O_{2x}$ supported at $x$) gives precisely
\eqref{mess}. \medskip

The only drawback of the reordering is that \eqref{serre} does not have a finite limit as $q\to1$, which is why we must leave taking the limit until last. While each of the contributions (\ref{aSerre}, \ref{bSerre}, \ref{cSerre}) of (a), (b) and (c) have finite $q\to1$ limits, \eqref{serre} is a combination of parts of those contributions which do not.

\subsection{Hall algebra}
Consideration of spaces like \eqref{motive} suggest that we work with the Ringel-Hall algebra of Joyce \cite{JoyceII}. Namely, we allow ourselves to weight the stack
$\T$ of 0-dimensional sheaves by another stack mapping to it, such as 
\begin{itemize}
\item $\Hom(\I_C,\ \cdot\ )$, the stack whose fibre over $T\in\T$ is $\Hom(\I_C,T)$,
\item $\Onto(\I_C,\ \cdot\ )$, the stack whose fibre over $T\in\T$ is $\Onto(\I_C,T)$,
\item $1_\T$, the stack $\T$ mapping to $\T$ by the identity map.
\end{itemize}
(So we think of each sheaf $T$ as being weighted by $\Hom(\I_C,T)$, or its Serre polynomial $q^{h_T}$, in the first case, or simply 1 in the final case. The second example is in fact a scheme, since $\Aut(T)$ acts freely.)
And we want to be able to ``add" objects $T_1,T_2$ in $\T$ by taking all of the extensions $\Ext^1(T_2,T_1)$ between them and mapping this to $\T$ via the universal extension (so that $e\in\Ext^1(T_2,T_1)$ maps to the sheaf $T\in\T$ that is defined by the extension $e$). This (noncommutative!) addition is the convolution product in the Hall algebra of stacks over $\T$. As usual we also keep track of automorphisms, and carry the weights along with the ``addition". The upshot is the following.

Let $\T^{\,2}$ be the stack of all short exact sequences $0\to T_1\to T\to T_2\to 0$ in $\T$. Mapping such an extension to $T\in\T$ defines a map $\T^{\,2}\to\T$ which we define to be the product $1_\T*1_\T$ of the stack $\T$ with itself. For any other stacks $U\to\T$ and $V\to\T$ we form their product from this one by fibre product:
$$
\xymatrix{
U*V \rto\dto& 1_\T*1_\T \rto\dto^{(\pi_1,\pi_2)}& \T \\
U\times V \rto& \T\times\T.
}$$
That is, $U*V$ is the fibre product of $U\times V\to\T\times\T$ and
$(\pi_1,\pi_2)\colon\T^{\,2}\to\T\times\T$, where $(\pi_1,\pi_2)$ maps the sequence $0\to T_1\to T\to T_2\to 0$ to $(T_1,T_2)$.

The product $*$ is associative, essentially because $1_\T*1_\T*1_\T$ is the stack of filtrations $T_1<T_2<T_3$, independently of the order we do the product in. (Here $T_1$ lies in the first copy of $\T$, $T_2/T_1$ in the second, and $T_3/T_2$ in the third.) Its identity is $1_0$, the stack consisting of the zero sheaf with its inclusion in $\T$.

In this language, the inclusion-exclusion \eqref{incexc} becomes the identity
\begin{multline} \label{psi}
\qquad \Onto(\I_C,\ \cdot\ )=\Hom(\I_C,\ \cdot\ )-\Hom(\I_C,\ \cdot\ )*1_{\T'} \\ + \Hom(\I_C,\ \cdot\ )*1_{\T'}*1_{\T'}+\ldots\,, \qquad
\end{multline}
where $\T'=\T\take\{0\}$ is the stack of \emph{nonzero} 0-dimensional sheaves.
This is the inversion of Bridgeland's generalisation \cite{Br} of the Reineke formula
\beq{Reineke}
\Hom(\I_C,\ \cdot\ )=\Onto(\I_C,\ \cdot\ )*1_\T
\eeq
which simply says that any $\phi\in\Hom(\I_C,T)$ is a surjection $\I_C\to T_1:=\im\phi$ followed by an extension $0\to T_1\to T\to T/T_1\to0$.
Since $1_\T=1_0+1_{\T'}$ it can be inverted via
$$
1_\T^{-1}\ =\ 1_0\ -\ 1_{\T'}\ +\ 1_{\T'}*1_{\T'}\ -\ 1_{\T'}*1_{\T'}*1_{\T'}\ +\ \ldots\,.
$$
(Formally we need to pass to a completion of the Hall algebra to make sense of this formula, but in our applications a finite truncation will suffice.) Applied to \eqref{Reineke} this gives \eqref{psi} and \eqref{incexc}.

In particular then, as an example, the collection of spaces \eqref{motive}, as we allow $\O_y$ to vary over any nonzero sheaf in $\T$, is the stack
$$
\Hom(\I_C,\O_x)*1_{\T'}.
$$

Restricting to finite type algebraic stacks with affine
geometric stabilizers mapping to $\T$, Joyce produces a motivic Ringel-Hall algebra $H(\T)$
with the above properties, together with an integration map
$$
P_q\colon H(\T)\to\Q(q^{1/2})[t]
$$
which, for instance, takes any Zariski-locally trivial global quotient $A/G\subset\T_n$ to the quotient of Serre polynomials $(P_q(A)/P_q(G))t^n$. Here $\T_n\subset\T$ is the stack of 0-dimensional sheaves of length $n$, and the $t$ is a formal variable to keep
track of this length. (Of course Joyce works in much more generality than this particular example.)

In particular, then, $P_q$ applied to $\Onto(\I_C,\ \cdot\ )\to\T$ gives the series
\beq{MNOP}
P_q(\Onto(\I_C,\ \cdot\ ))(t)=\sum_nP_q(I_n(X,C))t^n=:Z^I_C(X)(q,t),
\eeq
with limit $Z^I_C(X)(t)=\lim_{q\to1}Z^I_C(X)(q,t)$ the usual MNOP generating series of Euler characteristics of Hilbert schemes $I_n(X,C)$. Similarly
\beq{SP}
P_q(\Pure(\ \cdot\ ,\O_C))(t)=\sum_ne(P_n(X,C))t^n=:Z^P_C(X)(q,t),
\eeq
with limit $Z^P_C(X)(t)=\lim_{q\to1}Z^P_C(X)(q,t)$ the usual stable pairs generating series of Euler characteristics of the space $P_n(X,C)$.

In fact by Serre duality and Riemann-Roch, $P_q$ is a \emph{Lie algebra homomorphism} to the abelian Lie algebra $\Q(q^{1/2})[t]$:

\begin{thm} \label{liealg}
$P_q(U*V)=P_q(V*U)$.
\end{thm}

\noindent\emph{``Proof"}.
This is just a restatement of \cite[Equations 80-83]{JoyceII}.
These equations appear in the body of the proof of \cite[Theorem 6.1]{JoyceII},
which only holds for abelian categories of global dimension $1$. But in fact this assumption is only used in the last equality of Equation 83, while Equations 80-82 and the first equality in Equation 83 hold in general, as Joyce explains (they are proved in \cite[Corollary 5.15]{JoyceII} and \cite[Proposition 5.14]{JoyceII}).

Intuitively, the result should be clear from things we have already seen. The extensions $\Ext^1(T_2,T_1)$ between $T_1$ in the image of $U\to\T$ and $T_2$ in the image of $V\to\T$ also have a symmetry $\Hom(T_2,T_1)$, so the left hand integral is an integral over $U\times V$ of $P_q(\Ext^1(T_2,T_1))/P_q(\Hom(T_2,T_1))
=q^{\ext^1(T_2,T_1)-\hom(T_2,T_1)}$.

Similarly the right hand side is an integral over $U\times V$ of
$q^{\ext^1(T_1,T_2)-\hom(T_1,T_2)}$. But since $T_1,T_2\in\T$ are 0-dimensional, their Mukai pairing $\chi(T_1,T_2)$ is zero. Therefore
\beq{T2T1}
\ext^1(T_2,T_1)-\hom(T_2,T_1)=\ext^1(T_1,T_2)-\hom(T_1,T_2),
\eeq
by Riemann-Roch and Serre duality. \hfill$\square$

\medskip
\begin{rmk}
It is important to point out that this theorem \emph{is not} Joyce's Lie algebra homomorphism \cite[Theorem 6.12]{JoyceII} for the category of coherent sheaves on a Calabi-Yau threefold.
That homomorphism is the derivative at $q=1$ of the above result, and only defined on the Lie subalgebra of virtual indecomposables, making it much more complicated. Theorem
\ref{liealg} is a much simpler statement for the category $\mathcal{T}$, essentially saying that the Mukai pairing vanishes. When $\T$ is replaced by more general 
stacks of sheaves and $U$ and $V$ map to sheaves with constant Mukai pairing $m$, then
the result becomes $P_q(U*V)=q^mP_q(V*U)$.
\end{rmk} 

So this Lie algebra homomorphism is a natural reflection of Serre duality and Riemann-Roch. But, remarkably, Kontsevich and Soibelman have found a way to tweak the integration map to lift it to an \emph{algebra homomorphism}. This uses the Behrend $\chi^B$-weighting \cite{BehrendDT} in a crucial way, and is still conjectural in parts. This theory is clearly very powerful but we do not use it as we do not pretend to understand it. Instead we use the following weaker result from Joyce's theory.

\begin{thm} \label{lim*}
If both of $\lim_{q\to1}P_q(U)$ and $\lim_{q\to1}P_q(V)$ exist, then
$$
\lim_{q\to1}P_q(U*V)=\lim_{q\to1}P_q(U)\lim_{q\to1}P_q(V).
$$
\end{thm}

\noindent\emph{``Proof"}.
Again this follows directly from \cite[Equations 80-83]{JoyceII} on taking the limit
$q\to1$. Conceptually it is very simple. As discussed above, the left hand side is an integral of $q^{\ext^1(T_2,T_1)-\hom(T_2,T_1)}$ over $U\times V$. But this has limit $1$ as $q\to1$. \hfill$\square$

\medskip
At first sight this seems to say that on passing to the limit $q\to1$, all extension information is lost. Of course this is not the case, since for most $U$ to which we apply $P_q$, automorphisms mean that $\lim_{q\to1}P_q(U)$ does not exist. 

\subsection{First proof of the topological DT/PT wall crossing} We now have the tools to prove Theorem \ref{EulerPT}. We give a proof along the lines that we have been discussing, for completeness. In the next section we give a shorter and simpler proof.

As in \eqref{MNOP},
$$
Z^I_C(X)(q,t)=P_q\big(\Onto(\I_C,\ \cdot\ )\big)(t),
$$
which by \eqref{Reineke} gives
\beq{ZI}
Z^I_C(X)(q,t)=P_q\big(\Hom(\I_C,\ \cdot\ )*1_{\T}^{-1}\big)(t).
\eeq
Setting $C=\emptyset$, i.e. $\I_C=\O_X$, we recover the degree 0 generating series
\beq{Z0}
Z^I_0(X)(q,t)=P_q\big(\C^{[\,\cdot\,]}*1_{\T}^{-1}\big)(t).
\eeq
Here $\C^{[\,\cdot\,]}$ is the stack whose fibre over $T\in\T$ is $\C^{[T]}$, where $[T]$ is the length of $T$. Over strata of $\T$ on which $\hom(\O_X,T)=[T]$ is
constant, $\Hom(\O_X,T)$ is Zariski-locally equivalent to $\C^{[T]}$. Therefore they
have the same integral.

The analogue of Reineke's formula \eqref{Reineke} for $\Pure$ is
$$
\Ext^1(\ \cdot\ ,\O_C)=1_\T*\Pure(\ \cdot\ ,\O_C).
$$
This just says that any extension from $T\in\T$ to $\O_C$ is uniquely the composition of a quotient $T\to T_1\to0$ and a \emph{pure} extension from $T_1$ to $\O_C$ -- i.e. one which does not factor through any further quotient of $T_1$.

Therefore,
$$
Z^P_C(X)(q,t)=P_q\big(1_{\T}^{-1}*\Ext^1(\ \cdot\ ,\O_C)\big)(t).
$$
By Theorem \ref{liealg} (i.e. the Serre duality and Riemann-Roch formula \eqref{T2T1}), this is
$$
P_q\big(\Ext^1(\ \cdot\ ,\O_C)*1_{\T}^{-1}\big)(t).
$$

By Riemann-Roch and Serre duality for $\I_C$ \eqref{RR}, the stacks $\Hom(\I_C,\ \cdot\ )$ and $\Ext^1(\ \cdot\ ,\O_C)\oplus\C^{[\,\cdot\,]}$ are both Zariski-locally trivial of the same rank over strata of $\T$ on which $\hom(\I_C,\ \cdot\ )$ is constant.
Therefore they have the same integral $P_q$. Now $\C^{[T]}$ has Serre polynomial $q^{[T]}$ and contributes to the coefficient of $t^{[T]}$, suggesting the substitution of $qt$ for $t$ in the above formula. So we consider
\begin{eqnarray} \nonumber
Z^P_C(X)(q,qt) &=& P_q\big(\Ext^1(\ \cdot\ ,\O_C)*1_{\T}^{-1}\big)(qt) \\
&=& P_q\Big(\big(\Ext^1(\ \cdot\ ,\O_C)\oplus\C^{[\,\cdot\,]}\big)*
\big(\C^{[\,\cdot\,]}\big)^{-1}\Big)(t) \nonumber \\ \label{ZPC} &=&
P_q\Big(\Hom(\I_C,\ \cdot\ )*\big(\C^{[\,\cdot\,]}\big)^{-1}\Big)(t).
\end{eqnarray}
We compare this to \eqref{ZI},
\begin{eqnarray*}
Z^I_C(X)(q,t) &=& P_q\Big(\Hom(\I_C,\ \cdot\ )*\big(\C^{[\,\cdot\,]}\big)^{-1}*
\C^{[\,\cdot\,]}*1_{\T}^{-1}\Big)(t) \\ &=& P_q(U*V)(t),
\end{eqnarray*}
where $U=\Hom(\I_C,\ \cdot\ )*\big(\C^{[\,\cdot\,]}\big)^{-1}$ and
$V=\C^{[\,\cdot\,]}*1_{\T}^{-1}$.

By \eqref{ZPC}, $\lim_{q\to1}P_q(U)$ exists and equals $Z^P_C(X)$. By \eqref{Z0}, $\lim_{q\to1}P_q(V)$ exists and equals $Z^I_0(X)$. Therefore by Theorem \ref{lim*} we find that
$$
Z^I_C(X)(t)=\lim_{q\to1}\Big(Z^I_C(X)(q,t)\Big)=Z^P_C(X).Z^I_0(X),
$$
as required.

\begin{rmk} To prove the punctual statement \eqref{punctual} it is enough to replace the category $\mathcal{T}$ in the argument above by the category $\mathcal{T}(p)$ of $0$-dimensional sheaves supported at $p \in C$ (that $C \subset \mathbb{A}^3$ is now affine causes no difficulties). 
\end{rmk}

\subsection{Insertions}
We note that Theorem \ref{main} does not require the 3-fold to be Calabi-Yau. This comes
as a surprise to those of us brought up on virtual cycles. The virtual dimensions of
$I_n(X,\beta)$ and $P_n(X,\beta)$ are both $\int_\beta c_1(X)$, so to get invariants
from the
virtual class when this is strictly positive we must use insertions. (We do not currently
understand how to fit descendants into this picture.) In particular one does not expect
identities that use the whole moduli space, as in Thereom \ref{main}.

However, we are grateful to Rahul Pandharipande for pointing out that one can incorporate
insertions into our result.
Since Theorem \ref{EulerPT} is true for each curve $C$ at a time, we can manipulate
it into a form more reminiscent of the DT/PT conjectures in the case that $\int_\beta
c_1(X)>0$. 

Let $\pi_I\colon X\times I_{n}(X,\beta)
\to I_{n}(X,\beta)$ and $\pi_P\colon X\times P_{n}(X,\beta)
\to P_{n}(X,\beta)$ be the obvious projections, and let $\pi_X$ denote the projection
from either product to $X$. Let $\mathcal Z\subset X\times I_{n}(X,\beta)$ denote the
universal subscheme, and let $\mathbb F$ be the universal sheaf over $X\times P_{n}(X,\beta)$
\cite{PT1}.

From any $T\in H^d(X,\Z)$ we define a cohomology class $\mu(T)\in H^{d-2}(I_{n}(X,\beta))$
by
$$
\mu(T):=\pi_{I*}\big(\pi_X^*(T)\cup\text{ch}_2(\O_{\mathcal Z})\big),
$$
and similarly $\mu(T)\in H^{d-2}(P_{n}(X,\beta))$ by
$$
\mu(T):=\pi_{P*}\big(\pi_X^*(T)\cup\text{ch}_2(\mathbb F)\big).
$$
Fix $\mathbf T=(T_i\in H^{d_i}(X,\Z))_{i=1}^k$ of total degree
$\sum_{i=1}^kd_i=\int_\beta c_1(X)+2k$. We define the invariants with insertions $\mathbf
T$ by integrating over the virtual cycle:
$$
I^{vir}_{n,\beta}(\mathbf T):=\int_{[I_n(X,\beta)]^{vir}}\mu(T_1)\cup\ldots\cup\mu(T_k), $$
and
$$
P^{vir}_{n,\beta}(\mathbf T):=\int_{[P_n(X,\beta)]^{vir}}\mu(T_1)\cup\ldots\cup\mu(T_k).
$$
These invariants are then conjectured in \cite{PT1} to satisfy the identity
$$
I_{n,\beta}^{vir}(\mathbf T)\ =\ P^{vir}_{n, \beta}(\mathbf T)\ +\ I^{vir}_{1,0}\cdot
P^{vir}_{n-1, \beta}(\mathbf T)\ +\ I^{vir}_{2,0}\cdot P^{vir}_{n-2,\beta}(\mathbf T)\
+\ \ldots\ .
$$

Dually they can be defined by homology classes $T_i$ on $X$ and the slant product in
place
of the cup product. Capping the resulting cohomology classes with the virtual class loosely
corresponds to cutting down the virtual cycle to those subschemes (respectively pairs)
which intersect all of the cycles $T_i$. So if we let $\mathbf T$ denote the cycle of
those Cohen-Macaulay curves which intersect all of the cycles $T_i$ then we think of the insertion invariants as the intersection of the virtual cycle with the pullback via $\varphi_I$
(respectively $\varphi_P$) of $\mathbf T$.

In place of intersection with the virtual cycle we take as our topological analogue the Euler characteristic instead:
$$
I_{n,\beta}(\mathbf T):=e(\varphi_I^*\mathbf T),\qquad\text{and}\qquad P_{n,\beta}(\mathbf
T):=e(\varphi_P^*\mathbf T).
$$
Pushing down by $\varphi_I$ we see that $e(\varphi_I^*\mathbf T)$ is the Euler characteristic of $\mathbf T$ weighted by the constructible function $I_{n,C}$. (Similarly for
$e(\varphi_P^*\mathbf T)$.) Therefore integrating Theorem \ref{EulerPT} over $\mathbf
T$ gives the identity
$$
I_{n,\beta}(\mathbf T)\ =\ P_{n, \beta}(\mathbf T)\ +\ I_{1,0}\cdot P_{n-1, \beta}(\mathbf
T)\ +\ I_{2,0}\cdot P_{n-2,\beta}(\mathbf T)\ +\ \ldots\ .
$$

With hindsight one can see these identities in the toric computations of \cite{PT2}.
That such analogies should be true more generally perhaps suggests that there might be
an extension
of Behrend's result for general 3-folds, expressing insertion invariants as weighted
Euler characteristics of cut down moduli spaces.

% Then any $T\in H^j(X,\Z)$ defines operators
% $$
% \widehat T:=\pi_{I*}\big(\pi_X^*(T)\cdot \text{ch}_2(\O_{\mathcal Z})
% \cap(\pi_I^*(\ \cdot\ ))\big)\colon  H_*(I_{n}(X,\beta))\to H_*(I_{n}(X,\beta))
% $$
% and
% $$
% \widehat T:=\pi_{P*}\big(\pi_X^*(T)\cdot \text{ch}_2(\mathbb F)
% \cap(\pi_P^*(\ \cdot\ ))\big)\colon  H_*(P_{n}(X,\beta))\to H_*(P_{n}(X,\beta))
% $$
% of degree $2-j$. Therefore if $T_i\in H^{j_i}(X,\Z)$ then
% $$
% \widehat T_k\circ\widehat T_{k-1}\circ\ldots\circ\widehat T_1
% $$
% has degree $2k-\sum_{i=1}^kj_i$. When this equals $-\int_\beta c_1(X)$ we can apply it
% to the virtual class to get the (MNOP or stable pairs) invariant associated to the
% insertions $T_1,\ldots, T_k$. For fixed insertions, these invariants are conjectured
% to satisfy the usual identity \eqref{virid}.

\section{A short proof of the topological DT/PT wall crossing}\label{secondProof}

Let $P_C$ be be the stack (in fact projective scheme) of stable pairs supported on $C$, and $I_C$ be the subset of the Hilbert scheme consisting of 1-dimensional subschemes $Z$ such that $\I_C/\I_Z$ is 0-dimensional.

Consider the stack of 2-term complexes $\{\O_X\to F\}$, where $F$ is supported on a one-dimensional
subscheme whose underlying Cohen-Macaulay curve is $C$, and the cokernel of the map is 0-dimensional. Filtering such a complex by either its maximal $0$-dimensional subsheaf or its maximal 0-dimensional quotient sheaf, this stack can be written as either of
$$
I_C*1_{\T[-1]}=1_{\T[-1]}*P_C.
$$
This is really the Harder-Narasimhan filtration for these complexes in the two different stability conditions on either side of the wall. Therefore
\beq{I1S}
I_C=1_{\T[-1]}*P_C*1_{\T[-1]}^{-1}.
\eeq
This is nothing but a convenient shorthand to encapsulate the wall crossing discussed in Section \ref{ptwc} (i.e. the triangles \eqref{ideal} and \eqref{pair}). This holds in the abelian category (of perverse sheaves containing $\T[-1]$, stable pairs and ideal sheaves) obtained by tilting the usual one with the torsion subcategory $\T$. Thus we can again apply Joyce's theory. Considered inside this abelian category, we denote $\T[-1]$ by $\S$, so that \eqref{I1S} gives
\beq{ending}
P_q(I_C)=P_q\big(1_{\S}*P_C*1_{\S}^{-1}\big).
\eeq

We would like to commute $1_{\S}$ past $P_C*1_{\S}^{-1}$ using Theorem \ref{liealg}. However $T[-1]\in\S$ has Mukai vector $\chi(T[-1],\ \cdot\ )=[T]$ with anything in $P_C*1_{\S}^{-1}$ so Theorem \ref{liealg} has to be modified accordingly. Namely, by \cite[Equations 80-83]{JoyceII},
$$
P_q\Big(1_{\S}*P_C*1_{\S}^{-1}\Big)=
P_q\Big(P_C*1_{\S}^{-1}*\C_{\S}^{[\,\cdot\,]}\Big).
$$
(Intuitively: the left hand side is the integral of $q^{\ext^1(T[-1],\ \cdot\ )-\hom(T[-1],\ \cdot\ )}$ over $\S\times
\big(P_C*1_{\S}^{-1}\big)$, the right hand side is the integral of $q^{[T]+\ext^1(\ \cdot\ ,T[-1])-\hom(\ \cdot\ ,T[-1])}$. But these are equal by Serre duality and Riemann-Roch.)

Therefore by Theorems \ref{lim*} and \ref{liealg}, the limit $q\to1$ of \eqref{ending} gives
\begin{eqnarray*}
Z_C^I(X) &=& \lim_{q\to1}P_q(P_C).\lim_{q\to1}P_q(1_{\T}^{-1}*\C_{\T}^{[\,\cdot\,]})
\\ &=& \lim_{q\to1}P_q(P_C).\lim_{q\to1}P_q(\C_{\T}^{[\,\cdot\,]}*1_{\T}^{-1}) \\ &=& Z_C^P(X).Z_0^I(X),
\end{eqnarray*}
by \eqref{Z0}. We explained in Section
\ref{gencase} how to integrate this identity over the space of $C$s to give Theorem \ref        {main},
i.e. $Z_\beta^I(X)=Z_\beta^P(X).Z_0^I(X)$.

\vspace{.5cm}
Max Planck Institute for Mathematics\\
{\tt stoppa@mpim-bonn.mpg.de}
\vspace{.5cm}\\
Department of Mathematics\\
Imperial College London\\
{\tt richard.thomas@imperial.ac.uk}
\end{document}